\documentclass[pdflatex]{amsart}
\usepackage{graphicx} 
\usepackage{color}
\usepackage{amsmath,amssymb,amsthm,mathrsfs}
\usepackage{tikz}
\usepackage{tikz-cd}
\usepackage{url}
\usepackage{hyperref}

\newcommand{\Z}{\mathbb Z}
\newcommand{\Q}{\mathbb Q}
\def\P{\mathbb P}
\newcommand{\C}{\mathbb C}
\newcommand{\F}{\mathbb F}
\newcommand{\A}{\mathbb A}
\newcommand{\G}{\mathbb G}
\newcommand{\p}{\mathfrak p}
\def\q{\mathfrak q}
\newcommand{\lam}{\lambda}
\newcommand{\HK}{\mathrm{K3}^{[2]}}
\newcommand{\HKn}{\mathrm{K3}^{[n]}}
\newcommand{\Hdg}{\mathrm{Hdg}}
\newcommand{\crys}{\mathrm{crys}}
\newcommand{\be}{\beta}
\newcommand{\al}{\alpha}
\newcommand{\et}{\mathrm{\acute et}}

\newcommand{\inj}{\hookrightarrow}
\newcommand{\ab}{\mathrm{ab}}
\newcommand{\pr}{\mathrm{pr}}
\newcommand{\id}{\mathrm{id}}
\newcommand{\cyc}{\mathrm{cyc}}
\newcommand{\ur}{\mathrm{ur}}
\newcommand{\sep}{\mathrm{sep}}
\newcommand{\adele}{ad{\grave{e}}le}
\newcommand{\idele}{id{\grave{e}}le}
\newcommand{\etale}{{\acute{e}}tale}

\DeclareMathOperator{\Spec}{Spec}
\DeclareMathOperator{\disc}{disc}
\DeclareMathOperator{\Pic}{Pic}

\DeclareMathOperator{\cl}{cl}
\DeclareMathOperator{\CH}{CH}
\DeclareMathOperator{\MT}{MT}
\DeclareMathOperator{\Art}{Art}
\DeclareMathOperator{\Aut}{Aut}
\DeclareMathOperator{\Nm}{Nm}
\DeclareMathOperator{\GL}{GL}
\DeclareMathOperator{\Res}{Res}
\DeclareMathOperator{\End}{End}
\DeclareMathOperator{\rank}{rank}
\DeclareMathOperator{\ch}{ch}
\DeclareMathOperator{\Gal}{Gal}

\DeclareMathOperator{\Frac}{Frac}
\DeclareMathOperator{\Emb}{Emb}
\DeclareMathOperator{\Tr}{Tr}
\DeclareMathOperator{\Hom}{Hom}

\theoremstyle{plain}
 \newtheorem{thm}{Theorem}[section]
 \newtheorem{lem}[thm]{Lemma}
 \newtheorem{prop}[thm]{Proposition}
 \newtheorem{cor}[thm]{Corollary}
 
 \theoremstyle{definition}
 \newtheorem{rmk}[thm]{Remark}
 \newtheorem{dfn}[thm]{Definition}
 \newtheorem{exa}[thm]{Example}

\newcommand{\gmail}{\url{yamagishi.temma.m@gmail.com}}
\newcommand{\tusnummail}{\url{6125536@ed.tus.ac.jp}}
\newcommand{\tusaddressfull}{Department of Mathematics, Faculty of Science and Technology, Tokyo University of Science, 2641 Yamazaki, Noda, Chiba, 278-8510, Japan}

\title[On the Supersingular reduction of Cubic fourfolds with CM]{On the Supersingular reduction of Cubic fourfolds with Complex Multiplication}
\author{Temma Yamagishi}
\date{October 28,2026}
\address{\tusaddressfull}
\email{\gmail}
\email{\tusnummail}

\begin{document}
\begin{abstract}
 If $X$ is a K3 surface with complex multiplication by a CM field $E$, Kazuhiro Ito proved that the height of the reduction of $X$ in characteristic $p>0$, and its Artin invariant if it is supersingular, can be computed from the decomposition of $p$ in $E$ \cite[Theorem 1.1,1.2]{ItoK}. We present analogous results for cubic fourfolds by using the main theorem of complex multiplication for cubic fourfold, proved by Rikuto Ito \cite[Theorem 5.1]{ItoR}.
\end{abstract}

\maketitle

\tableofcontents

\section{Introduction}
\label{sec1}
\subsection{In the case of K3 surfaces} 
\label{sec1.1}
It is well known that the property ``complex multiplication'' (CM) is often useful when studying K3 surfaces. Let $S_{\C}$ be a smooth K3 surface over $\C$ which has a CM by $E$, i.e., $S_{\C}$ is a K3 surface whose endomorphisms preserve its $\Q$-Hodge structure as $1$-dim $E$-vector space. It is known that K3 surface over $\C$ has a model $S_K$ which is defined over number field $K \subset \C$. Fix a finite place $v$ of $K$ whose residue characteristic is $p>0$. If $S_{K}$ has a potential good reduction at $v$ and denote it by $\mathscr{S}_{\overline{v}}$ (this is K3 surface in positive characteristic), Kazuhiro Ito proved following two theorems \cite[Theorem 1.1, 1.2]{ItoK} using this property;

 Let $\q$ be a finite place of totally real number field $F \subset E$ below $v$.  
\begin{thm}\cite[Theorem 1.1]{ItoK} \label{1.1 }
\begin{enumerate}
\item[(1)] If $\q$ splits in $E$, the Picard number of $\mathscr{S}_{\overline{v}}$ is equal to $22-[E:\Q]$, and the height of the formal Brauer group of $\mathscr{S}_{\overline{v}}$ is $[E_{\p}:\Q_p]$.
\item[(2)] If $\q$ does not split in $E$
(i.e.\ $\q$ is either inert or ramified in $E$),
the Picard number of $\mathscr{S}_{\overline{v}}$ is $22$, and the height of the formal Brauer group of $\mathscr{S}_{\overline{v}}$ is $\infty$. In this case, \ $\mathscr{S}_{\overline{v}}$ is a supersingular $K3$ surface over $\overline{k(v)}$. 
\end{enumerate}
\end{thm}

\begin{thm}\cite[Theorem 1.2]{ItoK} \label{1.2}
Assume that $\q$ does not split in $E$, i.e.\ the $K3$ surface $\mathscr{S}_{\overline{v}}$ is supersingular. Moreover, we assume that both of the following conditions hold:
\begin{itemize}
\item The discriminant $\disc \Pic (S_{\C})$ of the lattice $\Pic (S_{\C})$ is not divisible by $p$.
\item The isomorphism $E(S_{\C}) \simeq E$ induces an isomorphism
\[
\mathrm{End}_{\Hdg}(T(S_{\C}))\otimes_{\Z}{\Z_{(p)}}\simeq \mathscr{O}_{E}\otimes_{\Z}{\Z_{(p)}},
\]
where $\mathscr{O}_{E}$ is the ring of integers of $E$.
\end{itemize}
Then the Artin invariant of $\mathscr{S}_{\overline{v}}$ is equal to $[k(\q):\F_{p}]$.
\end{thm}

This means that, we can calculate the important invariants for K3 surface which can be obtained from mod $p$ reduction of K3 surface over $\C$ with CM.

On the other hand, cubic fourfold, which is defined by homogeneous polynomial of degree 3 in $\P^5$, shares some similarities (e.g., Hodge structures, complex multiplication, lattice theory, etc.) with K3 surface. Therefore, we present analogous results for cubic fourfolds.

\subsection{Main results}
\label{sec1.2}
Our main results are analogues of Kazuhiro Ito's theorems for cubic fourfold with CM. Let $X_{\C}$ be a smooth cubic fourfold over $\C$ with a CM field $E$, i.e., $X_{\C}$ is a cubic fourfold whose endomorphisms preserve its $\Q$-Hodge structure is 1-dim $E$-vector space. It is known that cubic fourfold defined over $\C$ with CM has a model $X_K$ defined over number field $K \subset \C$. Fix a finite place $v$ of $K$ whose residue characteristic is $p>0$. If $X_{K}$ has a potential good reduction at $v$ and we denote it by $\mathscr{X}_{\overline{v}}$ (note that we need to assume $\mathscr{X}_{\overline{v}}$ is cubic fourfold in positive characteristic, see Remark \ref{1.5}). we proved following two theorems like the case of K3 surface.

\begin{thm}\label{1.3}
We assume that $p \geq 3$. 
\begin{enumerate}
\item[(1)] If $\q$ splits in $E$, the rank of  $\mathscr{X}_{\overline{v}}$ is equal to $23-[E:\Q]$, and the height of $\mathscr{X}_{\overline{v}}$ is $[E_{\p}:\Q_p]$. Here, the ``rank'' means the algebraic part of $H^4_{\crys} (\mathscr{X}_{\overline{v}}/W)$, and ``height'' means the slope of the $F$-crystal, see section \ref{sec2.2}.
\item[(2)] If $\q$ does not split in $E$
(i.e.\ $\q$ is either inert or ramified in $E$),
the rank of $\mathscr{X}_{\overline{v}}$ is $23$, and the height of $\mathscr{X}_{\overline{v}}$ is $\infty$. In this case, \ $\mathscr{X}_{\overline{v}}$ is a supersingular cubic fourfold over $\overline{k(v)}$. 
\end{enumerate}
\end{thm}

\begin{thm}\label{1.4}
Assume that $p \geq 7$ and $\q$ does not split in $E$, i.e.\ the cubic fourfold $\mathscr{X}_{\overline{v}}$ is supersingular. Moreover, we assume that both of the following conditions hold:
\begin{itemize}
\item The discriminant $\disc \CH^2 (X_{\C})$ of the lattice $\CH^2 (X_{\C})$ is not divisible by $p$.
\item The isomorphism $E(X_{\C}) \simeq E$ induces an isomorphism
\[
\mathrm{End}_{\mathrm{Hdg}}(T(X_{\C}))\otimes_{\Z}{\Z_{(p)}}\simeq \mathscr{O}_{E}\otimes_{\Z}{\Z_{(p)}},
\]
where $\mathscr{O}_{E}$ is the ring of integers of $E$.
\end{itemize}
Then the Artin invariant of $\mathscr{X}_{\overline{v}}$ is equal to $[k(\q):\F_{p}]$.
\end{thm}

The differences from \cite{ItoK} are the condition on the \emph{rank}, which corresponds to the Picard number of K3 surface, and the characteristic of $\overline{k(v)}$ satisfying $\ch \overline{k(v)}  = p \geq 3$ (resp. $p \geq 7$). Except for this differences, these can be proved using a similar approach to \cite{ItoK}. 

\begin{rmk}\label{1.5}
For a projective K3 surface $S_{K}$ over a number field $K \subset \C$ such that $S_{\C} := S_K \otimes_K \C$ has CM, it is conjectured that $S_{K_v}$ has potential good reduction at every finite place $v$ of $K$. Matsumoto proved that $S_{K_v}$ has potential good reduction when $S_{K_v}$ satisfies some conditions \cite{Matsumoto}.

However, for a cubic fourfold $X_{K}$, it is not clear what the necessary conditions for this are. Therefore, we assume that $X_{K_v}$ has potential good reduction and $\mathscr{X}_{\overline{v}}$ becomes cubic fourfold over $\overline{k(v)}$. 
\end{rmk}

\begin{rmk}\label{1.6}
Yang also stated in \cite[Remark (5.2.9)]{Yang} that, using explicit formulas of cubic fourfolds and their transcendental lattices by  \cite[Theorem 1.8]{LazaZheng}, we can compute the range of primes for cubic fourfolds which have supersingular reduction (with $\sigma_0 = 1$). Therefore, if the cubic fourfold with complex multiplication exists, this conjecture can be proven (without needing to assume that $\sigma_0 = 1$).
\end{rmk}

\subsection*{Notations}
\label{sec1.3}
In this paper, the term ``cubic fourfold'' (resp. ``K3 surface'') means only smooth one (resp. smooth projective) . Let $K$ be a field and we denote cubic fourfold (resp. K3 surface) over $K$ by $X_K$ (resp. $S_{K}$).  Let $F(X_K)$ be a Fano variety of lines which is $\HK$-type variety. Here $\HKn$-type variety means in the sense of \cite{Yang}. We define $\CH^i (X_K)$ to be the (numerical) Chow group of codimension $i$ cycles on $X_K$. 

If $K$ is a field , then $\overline{K}$ is algebraic closure of $K$. We denote the ring of Witt vectors of $K$ by $W:=W(K)$.  Let $p$ be a prime number and $\ch K = p$ , and we also assume $p \geq 7$ if cubic fourfold $X_{K}$ is defined over $K$ to make $H^i_{\crys} (F(X_K)/W)$ torsion free.

Let $F$ be a totally real number field, and $E$ be a totally imaginary quadratic extension of $F$ (i.e., $[E:F] = 2$). Fix a prime number $p$ and finite place $v$ of number field $K$ whose residue characteristic is $p$. Let
\begin{itemize}
\item $K_v$ be the completion of $K$ at $v$. 
\item $\mathscr{O}_{K_v}$ be the valuation ring of $K_v$. 
\item $k(v)$ be the residue field of $K_v$.
\end{itemize}
Moreover, consider the following extension
\[
\Q \subset F \subset E \subset K \subset \C.
\]
And let $\p$ be the finite place of $E$ and $\q$ be the finite place of $F$ below $v$.
We denote $E_{\p},\mathscr{O}_{E_{\p}},k(\p)$ and $F_{\q},\mathscr{O}_{E_{\q}},k(\q)$ by above meaning.

The term ``$X_{K}$ (resp. $S_{K}$) has potential good reduction at $v$'' means, after replacing $K_v$ by
a finite extension of it, there exists a proper smooth algebraic space $\mathscr{X}$ (resp. $\mathscr{S}$) over $\Spec \mathscr{O}_{K_v}$ whose generic fiber $\mathscr{X} \otimes_{\mathscr{O}_{K_v}} K_v$ (resp. $\mathscr{S} \otimes_{\mathscr{O}_{K_v}} K_v$) is isomorphic to $X_{K_v} := X_K \otimes_K K_v $ (resp. $S_{K_v} := S_K \otimes_K K_v $). In this paper, we assume that there exists such an algebraic space $\mathscr{X}$ (resp. $\mathscr{S}$) over $\Spec_{\mathscr{O}_{K_v}}$, and study the geometric special fiber of $\mathscr{X}$ (resp. $\mathscr{S}$), which is denoted by
\[
\mathscr{X}_{\overline{v}} := \mathscr{X} \otimes_{\mathscr{O}_{K_v}} \overline{k(v)}, \ \ \ \ (\mathrm{resp}.\ \mathscr{S}_{\overline{v}} := \mathscr{S} \otimes_{\mathscr{O}_{K_v}} \overline{k(v)}).
\]

Note that $\mathscr{X} \to \Spec \mathscr{O}_{K_v}$ automatically becomes projective because there exists an ample line bundle $\omega^{-1} \simeq \mathcal{O}(3)$ on  $\mathscr{X}_{\overline{v}}$. Therefore, we may assume that $\mathscr{X}$ is the scheme.

\subsection*{Acknowledgments}
\label{sec1.4}
The author would like to express his sincere gratitude to his advisor, Yuya Matsumoto, for his guidance and support throughout this work. The author is grateful to Rikuto Ito for the discussions on cubic fourfold and related topics. The author also thanks Kazuhiro Ito, who wrote the paper \cite{ItoK} that served as the inspiration for this work. In addition, his comments and suggestions are very useful for the author. 

\section{Preliminaries 1}

\label{sec2}
\subsection{CM for K3 surfaces and cubic fourfolds}
\label{sec2.1}
 We first study the complex multiplication in the case of K3 surface $S_{\C}$. Let $\Pic (S_{\C})$ be a Picard group of K3 surface $S_{\C}$ over $\C$. Consider the cycle map 
\[
\cl:  \Pic (S_{\C}) \to H^2 (S_{\C},\Z (1)) .
\]  

This map is injective, and the image of $\cl \otimes_{\Z} \Z (-1)$ coincides with $H^{1,1} \cap  H^2 (S_{\C},\Z)$. And we can define \emph{transcendental part} of $S_{\C}$ as below: 

\begin{dfn}
\label{2.1}
Let $S_{\C}$ be a K3 surface over $\C$. We denote \emph{transcendental part} of $S_{\C}$ as
\[
T(S_{\C}) := \Pic (S_{\C})^{\perp} \subset H^2 (S_{\C},\Z(1)),
\] 
which has the $\Z$-Hodge structure ($\Z$-HS) coming from $H^2 (S_{\C},\Z (1))$. 
\end{dfn}

Next, we will define the complex multiplication of K3 surfaces. 

\begin{dfn}
\label{2.2}
The K3 surface $S_{\C}$ is said to have \emph{complex multiplication} (CM) by $E$, if the $\Q$-algebra
\[
E(S_{\C}) := \End_{\Hdg} (T(S_{\C})) \otimes_{\Z} \Q
\]
is isomorphic to $E$, and $T(S_{\C}) \otimes_{\Z} \Q$ is $1$-dim $E(S_{\C})$-vector space. Moreover, we refer to such an $E$ as CM field (of K3 surface $S_{\C}$).
\end{dfn}

We fix an isomorphism $E\simeq E(S_{\C})$. Taking the canonical morphism $E \inj \End_{\C}(H^{2,0}(S_{\C})) \simeq \C$, we consider $E$ a subfield of $\C$.

Complex multiplication can be rephrased using the \emph{Mumford-Tate group}. Let $G:= \MT(T(S_{\C}) \otimes_{\Z} \Q)$ be the \emph{Mumford-Tate group} of $T(S_{\C}) \otimes_{\Z} \Q$. Since the $\Q$-Hodge structure on $T(S_{\C}) \otimes_{\Z} \Q$ is weight 0, $G$ is isomorphic to the Hodge group of $T(S_{\C}) \otimes_{\Z} \Q(-1)$. Zarhin showed that $G$ is isomorphic to the following algebraic group over $\Q$ \cite[Theorem 2.3.1]{Zarhin}
\[
\ker (\Nm : \Res_{E/\Q}\G_m \to \Res_{F/\Q} \G_m). 
\]

We fix an embedding $\overline{K} \inj \C$. We have an action of $\Gal(\overline{K}/K)$ on the $\mathrm{\etale}$ cohomology.
\[
H^2_{\et} (S_{\overline{K}},\A_{f}(1)) := \left(\prod_{p} H^2_{\et} (S_{\overline{K}},\Z_{p}(1)) \right) \otimes_{\Z} \Q
\]
where $\A_f :=  (\prod_{p} \Z_{p}) \otimes_{\Z} \Q$ is the ring of finite $\mathrm{\adele}$s of $\Q$. The embedding $\overline{K} \inj \C$ gives an isomorphism
\[
H^2_{\et} (S_{\overline{K}},\A_{f}(1)) \simeq H^2 (S_{\C},\Z (1)) \otimes_{\Z} \A_{f}.
\]
It induces an action of $\Gal(\overline{K}/K)$ on $T(S_{\C}) \otimes_{\Z} \A_{f}$.

Rizov showed \emph{the main theorem of complex multiplication for K3 surfaces} (cf. Theorem \ref{2.3}). We denote the involution on the ring $\A_{E,f}$ of finite $\mathrm{\adele}$s of $E$ induced by the complex conjugation $c: E \to E$ by the same symbol $c$. Let $\Art_K : \A_{K}^{\times} \to \Gal (K^{\ab}/K)$ be  the Artin reciprocity map provided by global class field theory with Deligne's normalization, i.e. which sends uniformizers to lifts of geometric Frobenius elements. Let $\Nm_{\A_K/\A_E} : \A_{K}^{\times} \to \A_{E}^{\times}$ be the norm map induced by the embedding $E \inj K$. 
\begin{thm}\cite{Rizov}
\label{2.3}
The action of $\Gal(\overline{K}/K)$ on $T(S_{\C}) \otimes_{\Z} \A_f$ factors through a continuous homomorphism
\[
\rho: \Gal(K^{\ab}/K) \to G(\A_f) = \{x \in \A^{\times}_{E,f}\ |\ xc(x) = 1\}.
\]
Moreover , the following diagram commutes: 
\[\begin{tikzcd}
	{\A^{\times}_{K}} && {\A^{\times}_{E}} \\
	{\Gal(K^{\ab}/K)} && {\A^{\times}_{E,f}} \\
	{G(\A_f)} && {G(\A_f)/G(\Q)}
	\arrow["{\Nm_{\A_K/\A_{E}}}", from=1-1, to=1-3]
	\arrow["{\Art_{K}}"', from=1-1, to=2-1]
	\arrow["\mathrm{pr}", from=1-3, to=2-3]
	\arrow["\rho"', from=2-1, to=3-1]
	\arrow["{y \mapsto c(y)y^{-1}}", from=2-3, to=3-3]
	\arrow[from=3-1, to=3-3]
\end{tikzcd}\]
where $G(\Q):=  \{x \in E^{\times} \ |\ xc(x) = 1\}$.
\end{thm}
\begin{proof}
See \cite[Corollary 3.9.2]{Rizov} and \cite[Theorem 12]{Taelman}. 
\end{proof}

On the other hand, we can also define same properties and theorem for cubic fourfold due to Rikuto Ito. To see that, we first define it as following :
\begin{dfn}
\label{2.4}
A \emph{cubic fourfold} over $k$ is smooth projective hypersurface of degree $3$ in $\P^5_k$.
\end{dfn}

 Let $\CH^2 (X_{\C})$ be the 2nd Chow group of cubic fourfold $X_{\C}$ over $\C$. Consider the cycle map 
\[
\cl^2 : \CH^2 (X_{\C}) \to H^4 (X_{\C},\Z (2)) .
\]  

This map is injective, and $A(X_{\C})$, which is the image of $\cl^2 \otimes_{\Z} \Z (-2)$, coincides with $H^{2,2} \cap H^4 (X_{\C},\Z)$ by Hodge conjecture for cubic fourfold \cite{Zucker}. And we can define \emph{transcendental part} and \emph{complex multiplication} of $X_{\C}$ as follows: 

\begin{dfn}
\label{2.5}
Let $X_{\C}$ be a cubic fourfold over $\C$. We denote \emph{transcendental part} of $X_{\C}$ as
\[
T(X_{\C}) := A (X_{\C})^{\perp} \subset H^4 (X_{\C},\Z(2)),
\] 
which has the $\Z$-HS coming from $H^4 (X_{\C},\Z (2))$. 

Let  $F$ be a totally real number field, and $E$ be a totally imaginary quadratic extension of $F$. The cubic fourfold $X_{\C}$ is said to have \emph{complex multiplication} (CM) by $E$, if the $\Q$-algebra
\[
E(X_{\C}) := \End_{\Hdg} (T(X_{\C})) \otimes_{\Z} \Q
\]
is isomorphic to $E$, and $T(X_{\C}) \otimes_{\Z} \Q$ is $1$-dim $E(X_{\C})$-vector space. Moreover, we refer to such an $E$ as CM field (of cubic fourfold $X_{\C}$). 
\end{dfn}

\begin{exa}
\label{2.6}
The simplest example of cubic fourfold $X_k$ is the Fermat type defined by 
\[
X_{0}^3 + X_{1}^3 +X_{2}^3 +X_{3}^3 +X_{4}^3 +X_{5}^3 = 0. 
\]
If $k = \C$, it is known that $X_{\C}$ has a CM field by $\Q (\sqrt{-3})$. 
\end{exa}

We fix an isomorphism $E\simeq E(X_{\C})$. Taking the canonical morphism $E \inj \End_{\C}(H^{3,1}(X_{\C})) \simeq \C$, we consider $E$ a subfield of $\C$.

Let $G: = \MT (T(X_{\C}) \otimes_{\Z} \Q)$ be a \emph{Mumford-Tate group} of $T(X_{\C})$. Since $\Q$-HS on $T(X_{\C}) \otimes_{\Z} \Q$ is weight 0, $G$ is isomorphic to Hodge group of $T(X_{\C}) \otimes_{\Z} \Q (-2)$. We can show that $G$ is isomorphic to the following algebraic group over $\Q$ like the case of K3 surfaces \cite[Theorem 2.3.1, Remarks 2.6]{Zarhin}. 
\[
\ker (\Nm : \Res_{E/\Q} \G_m \to \Res_{F/\Q} \G_m)
\]

We fix an embedding $\overline{K} \inj \C$. We have an action of $\Gal(\overline{K}/K)$ on the $\mathrm{\etale}$ cohomology like the case of K3 surface.
\[
H^4_{\et} (X_{\overline{K}},\A_{f}(2)) := \left(\prod_{p} H^4_{\et} (X_{\overline{K}},\Z_{p}(2)) \right) \otimes_{\Z} \Q
\]
The embedding $\overline{K} \inj \C$ gives an isomorphism
\[
H^4_{\et} (X_{\overline{K}},\A_{f}(2)) \simeq H^4 (X_{\C},\Z (2)) \otimes_{\Z} \A_{f}
\]
It induces an action of $\Gal(\overline{K}/K)$ on $T(X_{\C}) \otimes_{\Z} \A_{f}$.

Rikuto Ito showed that \emph{the main theorem of complex multiplication for rank-21 cubic fourfold} like Rizov. Moreover, he generalize the theorem for cubic fourfolds of any rank.
\begin{thm}
\label{2.7}
Let $X_\C$ be a cubic fourfold over $\C$ with reflex field $E \subset \C$. Let $\sigma \in \Aut (\C/E)$ and $s \in \A^{\times}_{E}$ be an $\idele$ with $\Art_{E}(s) = \sigma|_{E^{\ab}}$. Then there exists a unique Hodge isometry $\eta : P^4 (X,\Q (2)) \to P^4 (X^{\sigma},\Q  (2))$ such that the following diagram 
\[\begin{tikzcd}
	{P^4(X,\Q(2))_{\A_f}} && {P^4(X^{\sigma},\Q(2))_{\A_f}} \\
	{P^4(X,\Q(2))_{\A_f}}
	\arrow["{\eta \otimes \A_f}", from=1-1, to=1-3]
	\arrow["{r_{X}(s)}", from=2-1, to=1-1]
	\arrow["\sigma"', from=2-1, to=1-3]
\end{tikzcd}\]
commutes.
\end{thm}
\begin{proof}
See \cite[Theorem 5.1, Corollary 5.8]{ItoR}
\end{proof}

He also reformulated the statement to give it a more algebraic meaning.
\begin{thm}
\label{2.8}
Let $X_{\C}$ be a cubic fourfold with the refrex field $E \subset K$. Suppose that $X$ has a model $X_0$ over $\Q$. Let $G$ be the Mumford-Tate group of $T(X)$. Then:

(1)\ The restriction $\rho|_{\Gal (\overline{\Q}/K)}$ factors as
\[
\Gal(\overline{\Q}/K) \to \Gal (K^{\ab}/K) \overset{\rho_0}\to G(\A_f) \inj \GL(T(X)_{\A_f}) 
\]
for a continuous homomorphism $\rho_0$.

(2)\ The diagram 
\[\begin{tikzcd}
	{\A^{\times}_K} && {\Gal(K^{\ab}/K)} \\
	{G(\A_f)} && {G(\A_{f})/G(\Q)}
	\arrow["{\Art_{K}}", from=1-1, to=1-3]
	\arrow["{r_X}"', from=1-1, to=2-1]
	\arrow["{\rho_0}", from=1-3, to=2-3]
	\arrow[from=2-1, to=2-3]
\end{tikzcd}\]
is commutative. 
\end{thm}
\begin{proof}
See \cite[Theorem 6.1]{ItoR}. Note that although he proved only the rank-$21$ cases, we can obtain the statement for cubic fourfolds of arbitrary rank replacing Theorem 5.1 to Corollary 5.8.
\end{proof}

Therefore, we can obtain a commutative diagram which coincides with Rizov's one
\begin{thm}
\label{2.9}
The action of $\Gal(\overline{K}/K)$ on $T(X_{\C}) \otimes_{\Z} \A_f$ factors through a continuous homomorphism
\[
\rho: \Gal(K^{\ab}/K) \to G(\A_f) = \{x \in \A^{\times}_{E,f}\ |\ xc(x) = 1\}.
\]
Moreover, the following diagram commutes: 
\[\begin{tikzcd}
	{\A^{\times}_{K}} && {\A^{\times}_{E}} \\
	{\Gal(K^{\ab}/K)} && {\A^{\times}_{E,f}} \\
	{G(\A_f)} && {G(\A_f)/G(\Q)}
	\arrow["{\Nm_{\A_K/\A_{E}}}", from=1-1, to=1-3]
	\arrow["{\Art_{K}}"', from=1-1, to=2-1]
	\arrow["\mathrm{pr}", from=1-3, to=2-3]
	\arrow["\rho"', from=2-1, to=3-1]
	\arrow["{y \mapsto c(y)y^{-1}}", from=2-3, to=3-3]
	\arrow[from=3-1, to=3-3]
\end{tikzcd}\]
\end{thm}
\begin{proof}
The former is already proven. For the latter, consider the following diagram: 
\[\begin{tikzcd}
	{\mathbb A^{\times}_{K}} & {\mathbb A^{\times}_{E}} & {\mathbb A^{\times}_{E,f}} \\
	{\Gal (K^{\ab}/K)} & {\Gal (E^{\ab}/E)} & {\mathbb A^{\times}_{E,f}/E^{\times}} \\
	{G(\mathbb A_{f})} && {G(\mathbb A_{f})/G(\mathbb Q)}
	\arrow["{\Nm_{\mathbb A_K /\mathbb A_{E}} }", from=1-1, to=1-2]
	\arrow["{\Art_K}"', from=1-1, to=2-1]
	\arrow["{\mathrm{pr}}", from=1-2, to=1-3]
	\arrow["{\Art_E}"', from=1-2, to=2-2]
	\arrow[from=1-3, to=2-3]
	\arrow[from=2-1, to=2-2]
	\arrow["\rho"', from=2-1, to=3-1]
	\arrow["CFT", from=2-2, to=2-3]
	\arrow[from=2-3, to=3-3]
	\arrow[from=3-1, to=3-3]
\end{tikzcd}\]
Since we can use the same argument as \cite{Taelman}, the lower square is commutative. The upper squares are commutative from their properties. So, we obtain that outside of this diagram is commutative.
\end{proof}

\subsection{Rank, height and supersingularity of cubic fourfold}
\label{sec2.2}
Let $X_{k}$ be a cubic fourfold over $k$. In this section, we define \emph{rank, height and supersingularity} of cubic fourfold. These are analogues of {Picard number, (slope) height, and supersingularity} of K3 surface.
\begin{dfn}
\label{2.10}
Consider the cycle map
\[
\cl^2 : \CH^2 (X_{\C}) \to H^4 (X_{\C},\Z (2)).
\]
If the rank of image $A(X_{\C})$ is $n$, we call $X_{\C}$ a \emph{rank-n} cubic fourfold. 
\end{dfn}

\begin{exa}
\label{2.11}
Fermat type in $\mathbb P^5_{\C}$: 
\[
X_{0}^{3} +X_{1}^{3}+X_{2}^{3}+X_{3}^{3}+X_{4}^{3} +X_{5}^{3} = 0 
\]
is rank-$21$ cubic fourfold over $\C$. 
\end{exa}

\begin{rmk}
\label{2.12}
Rank-$21$ cubic fourfold correspond to \emph{singular K3 surfaces}, which has Picard number $\rho = 20$. see also \cite{ItoR}
\end{rmk}

In positive characteristic, we define rank and Shioda supersingularity as follows. 

\begin{dfn}
\label{2.13}
Consider the cycle map
\[
\cl^2_{\crys} : \CH^2 (X_{k})\otimes W \to H^4_{\crys} (X_k /W).
\]
If the rank of this image $A(X)_k$ is $n$, we call $X_{k}$ a \emph{rank-n} cubic fourfold. Moreover, we call it (Shioda) supersingular if $n=23$.
\end{dfn}

We next define the \emph{height} of cubic fourfolds and Artin supersingularity. If we define the height based on the formal group like K3 surfaces, it does not become a 1-dimensional formal group, so that definition does not work. Therefore, we use \emph{slope} of the $F$-crystal given by the crystalline cohomology of cubic fourfold. Let $k$ be an algebraically closed field such that $\ch k =p \geq 7$ so that $H^n_{\crys} (F(X_k)/W)$ is torsion-free. We denote the $F : X_k \to X_k$ absolute Frobenius morphism. They induce $\varphi : H^4_{\crys} (X_{k}/W) \to H^4_{\crys} (X_{k}/W)$ which is a $\sigma$-linear map. The tuple 
\[
(H^4 (X_{k}/W), \varphi)
\]
defines the $F$-crystal structure. Note that this cannot be a K3 crystal which was defined by Ogus \cite[Definition 3.1]{Ogus} because its slope does not have suitable slopes. So we consider the twist
\[
(H^4_{\crys}(X_{k}/W (1)), \varphi)
\]
which is a K3 crystal that comes from $\HK$-type variety $H^2_{\crys} (F(X_k)/W)$. For details, see \cite[(5.2.4)]{Yang}. In particular, we can define the Artin invariant for cubic fourfold from this meaning. 

For the $F$-crystal coming from a smooth proper scheme over perfect field can define Newton polygon of the slope, See \cite[section 3]{Liedtke}.

\begin{dfn}
\label{2.14}

If $h \in \{1,\dots,10,11,\infty\}$ satisfies following conditions, it is called \emph{height} of cubic fourfold. 

(i)\  If the slopes of the Newton polygon of $H^4_{\crys} (X_k/W)$ are $2-\dfrac{1}{h},2,2+\dfrac{1}{h}$ , then the height is $h$. 

(ii)\ If the Newton polygon of $H^4_{\crys} (X_k /W)$ is a straight line (i.e., the slope is $2$), then $h = \infty$.
\end{dfn}

We denote the height of $X_k$ by $h(X_{k})$. 

Of the lines in the figure on the left below, the upper one corresponds to the supersingular case ($h=\infty$), while the lower one corresponds to ($h=1$). The figure on the right illustrates how the graph changes between these lines for each value of $h$.

\begin{figure}[htbp]
\begin{minipage}{0.48\columnwidth}
\begin{tikzpicture}
\label{picture}

\draw[-stealth](-1,-2)--(-1,3.5)node[left]{$y$};
\draw[-stealth](-2,-1.5)--(4.0,-1.5)node[right]{$x$};
\draw(-1,-1.5)node[below left]{O};

\draw [domain=-1:3] plot[smooth](\x,\x-0.5); 

\draw [domain=-1:0] plot[smooth](\x, 0.5*\x-1); 

\draw [domain=0:2] plot[smooth](\x, \x-1); 

\draw [domain=2:3] plot[smooth](\x, 1.5*\x-2); 

\draw[dashed](-1,-1)--(0,-1); 
\draw[dashed](0,-1)--(0,-1.5); 

\draw[dashed](-1,1)--(2,1); 
\draw[dashed](2,1)--(2,-1.5); 

\draw[dashed](-1,2.5)--(3,2.5); 
\draw[dashed](3,2.5)--(3,-1.5); 

\draw(0,-1.5)node[below]{$1$};
\draw(2,-1.5)node[below]{$22$};
\draw(3,-1.5)node[below]{$23$};

\draw(-1,-1)node[left]{$1$};
\draw(-1,1)node[left]{$43$};
\draw(-1,2.5)node[left]{$46$};
\end{tikzpicture}
\end{minipage}
\begin{minipage}{0.48\columnwidth}
\begin{tikzpicture}

\draw[-stealth](-1,-2)--(-1,3.5)node[left]{$y$};
\draw[-stealth](-2,-1.5)--(4.0,-1.5)node[right]{$x$};
\draw(-1,-1.5)node[below left]{O};

\draw [domain=-1:0.5] plot[smooth](\x, 0.75*\x-0.75); 

\draw [domain=0.5:1.5] plot[smooth](\x, \x-0.875); 

\draw [domain=1.5:3] plot[smooth](\x, 1.25*\x-1.25);

\draw[dashed](-1,-0.35)--(0.5,-0.35); 
\draw[dashed](0.5,-0.35)--(0.5,-1.5); 

\draw[dashed](-1,0.65)--(1.52,0.65); 
\draw[dashed](1.52,0.65)--(1.52,-1.5); 

\draw[dashed](-1,2.5)--(3,2.5); 
\draw[dashed](3,2.5)--(3,-1.5); 

\draw(3,-1.5)node[below]{$23$};
\draw(0.5,-1.5)node[below]{$h$};
\draw(1.5,-1.5)node[below]{$23-h$};

\draw(-1,2.5)node[left]{$46$};
\draw(-1,-0.35)node[left]{$2h-1$};
\draw(-1,0.65)node[left]{$45-2h$};

\end{tikzpicture}
\end{minipage}
\end{figure}

These two definitions of supersingularity are equivalent. Therefore, we will simply refer to them as supersingular cubic fourfold.

\begin{prop}
\label{2.15}
Let $\ch k= p  \geq 3$. The following are equivalent.

(i)\ $X_k$ is rank-$23$ cubic fourfold. (i.e. $X_k$ is Shioda supersingular)

(ii)\ $h (X_k) =\infty$. (i.e. $X_k$ is Artin supersingular)
\end{prop}

\begin{proof}
It follows from Tate conjecture for cubic fourfold. cf. \cite[Theorem 6.14]{Madapusi}
\end{proof}

Let $q$ be a power of a prime number $p$ and $X_{\F_q}$ be a cubic fourfold over $\F_q$ and $l \not= p$ be a prime number, and consider the base change and denote it $X_{\overline{\F_q}}$. Then we have a Galois action on $H^4_{\et} (X_{\overline{\F_q}}, \Q_{l}(2))$ by the geometric Frobenius element $\sigma \in \Gal (\overline{\F_{q}}/\F_{q})$. Let $P(T)$ be the characteristic polynomial of this action $\sigma$. Then, the degree of $P(T)$ is $23$ and the coefficient of $P(T)$ is $\Q$ by Weil conjecture.

Consider the roots of $P(T)$ and denote by $\al_1, \dots, \al_{23} \in \overline{\Q}$. Fix the embedding $\overline{\Q} \to \overline{\Q_p}$. Let $\nu : \overline{\Q_p} \to \Q \cup \{\infty\}$  be a $p$-adic valuation normalized by $\nu(p)= 1$. 

\begin{prop}{cf. \cite[Proposition 3.3]{ItoK}}
\label{2.16}

(1)\ The number of $i$ such that $\al_i$ is a root of unity is equal to $\rank A(X_{\overline{\F_{q}}})$. 

(2)\ If $h(X_{\overline{\F_q}}) < \infty$, then the number of $i$ such that $\nu(\al_i)>0$ is equal to $h(X_{\overline{\F_q}})$.

(3)\ If $h(X_{\overline{\F_q}}) = \infty$, then every $\al_i$ is a root of unity. Hence it satisfies $\nu(\al_i) =0$.
\end{prop}
\begin{proof}
The claim of rank follows from the Tate conjecture for cubic fourfolds \cite[Theorem 6.14]{Madapusi}. For height, see \cite[Proposition 6.17]{Liedtke} and \cite[Theorem 1]{Kats}. 
\end{proof}

We can define \emph{Artin invariant} of supersingular cubic fourfold. We assume $p \geq 7$.
\begin{dfn}
\label{2.17}
Let $X_k$ be a supersingular cubic fourfold. Then, the map 
\[
\cl^2_{\crys}: \CH^2 (X_k) \otimes \Z_p \overset{\sim}\to H^4_{\crys}(X_k/W)^{F=p^2}
\]
is isomorphic \cite[Proposition (5.2.2)]{Yang}. Since $H^4_{\crys} (X_{k}/W (1))$ has the structure of a supersingular K3 crystal, we can consider the discriminant group of this $\Z_p$-lattice, which is isomorphic to $(\Z/p\Z)^{2\sigma_0}$ (cf. \cite[Proposition 3.13]{Ogus}). We call this integer $\sigma_0$ as Artin invariant of $X_k$. 
\end{dfn}

\begin{exa}
\label{2.18}
The Fermat type cubic fourfold 
\[
X_{0}^3+X_{1}^{3}+X_{2}^3+X_{3}^{3}+X_{4}^3+X_{5}^{3} = 0
\]
is the supersingular cubic fourfold with $\sigma_0 = 1$ if $p \equiv -1 \pmod {3}$ \cite[Theorem 2.10, Proposition 3.5, Remark 3.8]{Shiodakatsura} and \cite[Lemma (5.2.6), Lemma (5.2.8), Example]{Yang}.
\end{exa}

\section{Proof of Theorem \ref{1.3}}
\label{sec3}
Let $\pi_{v}$ be a uniformizer of $K_v$. Since Artin reciprocity map sends $x := (1,\dots,1,\pi_v, 1 ,\dots,1) \in \A^{\times}_K$ to the lifts of geometric Frobenius, we have 
\[
\Art_{K}: \A^{\times}_{K} \to \Gal (K^{\ab}/K) ; \ \ \ \ x \mapsto \sigma|_{K^{\ab}},
\]
where $\sigma \in \Gal (\overline{K_v}/K_v)$. And $x$ maps to $z \in G(\A_f)$ where
\[
z = 
\begin{cases}
 (1,\dots,1, \Nm_{K_v/E_{\p}}(\pi_{v}) ^{-1},c \Nm_{K_v/E_{\p}} (\pi_v), 1 ,\dots,1)  & (\q: \text{split}) \\
 (1,\dots,1,c \Nm_{K_v/E_{\p}} (\pi_v) \Nm_{K_v/E_{\p}}(\pi_{v}) ^{-1}, 1 ,\dots,1) & (\q: \text{not split} )
\end{cases}
\]
Using Rikuto Ito's theorem \ref{2.9}, there exists $\eta \in G(\Q) \subset E^{\times}$ such that 
\[
\rho (\sigma|_{K^{\ab}}) = z \eta \in G(\A_f).
\]
By Proposition \ref{3.1}, it suffices to obtain the information of the roots of $\sigma$ to know the height. However, conjugates of $\eta$ can be described by whether $\q$ splits. 

\begin{lem}{\cite[Lemma 4.1]{ItoK}}
\label{3.1}

Fix an embedding $\overline{\Q} \inj \overline{\Q}_p$. Let $\nu : \overline{\Q}_p \to \Q \cup \{\infty\}$ be the $p$-adic valuation normalized by $\nu(p) = 1$ and $\Q(\eta)$ be the subfield generated by $\eta$ and $\p_0$ be the finite place of $\Q (\eta)$ below $\p$.

(1)\ If $\q$ splits in $E$, then the element $\eta$ is not a root of unity and the number of conjugations $\xi \in \overline{\Q}$ of $\eta$ over $\Q$ with $\nu(\xi) >0$ is equal to $[\Q(\eta)_{\p_0} : \Q _p]$. 

(2)\ If $\q$ does not split in $E$, then the element $\eta$ is a root of unity. 
\end{lem}
\begin{proof}
See \cite[Lemma 4.1]{ItoK}. 
\end{proof}

So we have to compute the eigenvalue of $\sigma$ acting on
\[
H^4_{\et}(X_{\overline{K_v}},\Q_l (2)) \simeq (A(X_{\C}) \otimes \Q_{l}) \oplus (T(X_{\C}) \otimes \Q_l (2)). 
\]

Since $T(X_{\C}) \otimes \Q$ is a 1-dim $E$-vector space, the characteristic polynomial of $\sigma$ acting on $T(X_{\C}) \otimes \Q_l$ can be identified with the characteristic polynomial of the $\Q$-linear map $E \to E; x \mapsto x \eta$. So, it is equal to $f(t)^{[E: \Q(\eta)]}$, where $f(t)$ is a minimal polynomial of $\eta$ over $\Q$.

If $\q$ splits in $E$, no root of $f(t)$ is a root of $1$. Let $\p_0 (\not= \p)$ be the finite place of $\Q(\eta)$ below $v$. We have
\[
[E: \Q(\eta)] = [E_{\p}: \Q(\eta)_{\p_0}].
\]

Since the element $\eta$ is not a root of unity and the number of conjugates $\xi \in \overline{\Q}$ of $\eta$ over $\Q$ with $\nu (\xi)> 0 $ is equal to $[\Q(\eta)_{\p_0} : \Q_p]$, the number of roots $\xi \in \overline{\Q}$ of $f(t)^{[E: \Q(\eta)]}$ with $\nu (\xi)>0$ is counted with their multiplication is equal to 
\[
[E: \Q(\eta)][\Q(\eta): \Q] = [E_{\p}: \Q(\eta)_{\p_0}][\Q(\eta)_{\p_0}:\Q_p ] = [E_{\p}: \Q_p].
\]

The eigenvalues of $\sigma$ acting on 
\[
A(X_{\C}) \otimes \Q_{l} \simeq A(X_{\overline{K_v}}) \otimes \Q_l
\]
are roots of $1$. Therefore, the number of eigenvalues $\xi \in \overline{\Q}$ of $\sigma$ acting on $H^4_{\et} (X_{\overline{K_v}},\Q_l (2))$ which are roots of $1$ is  equal to the rank of $A(X_{\C})$, counted with their multiplicities. Moreover, the number of eigenvalues $\xi \in \overline{\Q}$ of $\sigma$ acting on $H^4_{\et} (X_{\overline{K_v}},\Q_l (2))$ with $\nu (\xi) > 0$ is equal to $[E_{\p}:\Q_p]$, counted with their multiplicities. It coincides with the height of $\mathscr{X}_{\overline{v}}$.

If $\q$ does not split in $E$, every root of $f(t)$ is a root of $1$. Therefore, every eigenvalue of $\sigma$ acting on $H^4_{\et} (X_{\overline{K_v}},\Q_l (2))$ is a root of $1$. It follows that the rank of $\mathscr{X}_{\overline{v}}$ is equal to $23$. i.e., $\mathscr{X}_{\overline{v}}$ is supersingular. 

The proof of Theorem \ref{1.3} is complete.

\section{Preliminaries 2}
\label{sec4}
We keep the notation in the previous sections.

\subsection{Unramifiedness of $E_{\p}/F_{\q}$}
\label{sec4.1}

\begin{rmk}
\label{4.1}
If we replace $H^2_{\et} (S_{\overline{K_v}},\Z_p)$ in the case of the K3 surface with $H^4_{\et} (X_{\overline{K_v}},\Z_p)$ in the case of the cubic fourfold, the content of this section is almost described by \cite{ItoK}. Thus, we use similar method of \cite{ItoK}.
\end{rmk}

First, We may assume that the extension $E_{\p}/F_{\q}$ is unramified like \cite[Lemma 2.4]{ItoK}. To prove that, we can prove \cite[Lemma 2.2]{ItoK} for the cubic fourfold with CM.

\begin{lem}{\cite[Lemma 2.2]{ItoK}}
\label{4.2}
We assume that the isomorphism $E(X_{\C}) \simeq E$ induces an isomorphism
\[
\mathrm{End}_{\Hdg}(T(X_{\C})) \otimes_{\Z} \Z_{(p)} \simeq \mathscr{O}_{E}\otimes_{\Z} \Z_{(p)}. 
\]
Then $T(X_{\C})\otimes_{\Z} \Z_{(p)}$ is a free $\mathscr{O}_{E}\otimes_{\Z} \Z_{(p)}$-module of rank $1$. Moreover, if we fix an isomorphism of $\mathscr{O}_{E}\otimes_{\Z} \Z_{(p)}$-modules
\[
T(X_{\C})\otimes_{\Z} \Z_{(p)} \simeq \mathscr{O}_{E}\otimes_{\Z} \Z_{(p)}, 
\]
there exists an element $d \in F^{\times}$ such that the intersection product on $T(X_{\C})\otimes_{\Z} \Z_{(p)}$ is translated to the pairing 
\[
\langle\ ,\ \rangle : \mathscr{O}_{E}\otimes_{\Z} \Z_{(p)} \times\mathscr{O}_{E}\otimes_{\Z} \Z_{(p)} \to \Z_{(p)}
\]
defined by $\langle x, y \rangle=\mathrm{Tr}_{E/\Q}(dxc(y))$ for every $x, y \in \mathscr{O}_{E}\otimes_{\Z} \Z_{(p)}$.
\end{lem}
\begin{proof}
Since \cite[Proposition 1.3.10]{Chen}, which is used in \cite[Lemma 2.2]{ItoK}, also holds when the assumption is replaced to cubic fourfold. 
\end{proof}

\begin{prop}
\label{4.3}
Under the assumptions of Theorem \ref{1.4}, the quadratic extension $E_{\p}/F_{\q}$ is unramified.
\end{prop}
\begin{proof}
Using the Lemma \ref{4.2} and the assumption that $\disc A(X_{\C})$ is not divisible by $p$, $\mathscr{O}_{E} \otimes_{\Z} \Z_{(p)}$ has a (perfect) pairing $\langle\ , \ \rangle$, there is a decomposition 
\[
\mathscr{O}_{E} \otimes_{\Z} \Z_p \simeq \prod_{\p' | p} \mathscr{O}_{E_{\p'}},
\]
where $\p'$ is a finite place of $E$ above $p$. Therefore, we have a perfect pairing 
\[
\langle\ , \ \rangle_{\p} := \mathscr{O}_{E_{\p}} \times \mathscr{O}_{E_{\p}} \to \Z_p; \ \ \ \ (x,y) \mapsto \Tr_{E_{\p}/\Q_p}(dxc(y))
\]
where $c : E_{\p} \simeq E_{\p}$ is the involution induced by complex conjugation on $E$. 

Let $\mathcal{D}_{A/B}$ be the different of $A$ over $B$. Since $\langle\ , \ \rangle_{\p}$ is perfect, we see that $\mathcal{D}_{E_{\p}/\Q_p}$ is generated by $d^{-1} \in F^{\times}$ as an ideal of $\mathscr{O}_{E_{\p}}$. 

We assume that $E_{\p}/F_{\q}$ is totally ramified. Consider $\nu_{\p}: E_{\p}^{\times} \to \Z$ to be the valuation normalized $\nu_{\p} (\pi_{\p}) = 1$. A generator $e \in \mathscr{O}_{E_{\p}}$ of $\mathcal{D}_{E_{\p}/F_{\q}}$ satisfies $\nu_{\p}(e) = 1$ because $p>2$ \cite[Chapter 3, Proposition 13]{Serrelf}. By \cite[Chapter 3, Proposition 8]{Serrelf}, we have 
\[
\mathcal{D}_{E_{\p}/\Q_p} = \mathcal{D}_{E_{\p}/F_{\q}} \mathcal{D}_{F_{\q}/\Q_p} 
\]
as an ideals of $\mathscr{O}_{E_\p}$. Let $f \in \mathscr{O}_{F_{\q}}$ be a generator of $\mathcal{D}_{F_{\q}/\Q_p}$. Then we have 
\[
\nu_{\p} (d^{-1}) = \nu_{\p} (e) +\nu_{\p} (f) = 1+ \nu_{\p} (f).
\]
Since $d^{-1},f \in \mathscr{O}_{F_{\q}}$ and $E_{\p}/F_{\q}$ are totally ramified, $\nu_{\p}(d^{-1}),\nu_{\p}(f)$ are even. This is a contradiction.
\end{proof}

\subsection{Integral $p$-adic $\mathrm{\etale}$ cohomology}
\label{sec4.2}
In this section, we will describe the Galois module $H^4_{\et} (X_{\overline{K_v}}, \Z_p)$.  We assume that $X_{\C}$ is a cubic fourfold over $\C$ with CM field by $E$ satisfying the assumptions of Theorem \ref{1.4}, i.e. its discriminant is not divisible by $p$, and the isomorphism $E(X_{\C}) \simeq E$ induces an isomorphism
\[
\mathrm{End}_{\mathrm{Hdg}}(T(X_{\C}))\otimes_{\Z}{\Z_{(p)}}\simeq \mathscr{O}_{E}\otimes_{\Z}{\Z_{(p)}}.
\]

 In particular, suppose that $X_K$ has potential good reduction at $v$ and the reduction $\mathscr{X}_{\overline{v}}$ is a supersingular cubic fourfold over $\overline{k(v)}$. Thus, we assume that $\q$ does not split in $E$ (see Theorem \ref{1.3}).  

We shall define \emph{Lubin-Tate character}. For details,  see \cite[section 5]{ItoK}. Fix a uniformizer $\pi_{\p}$ of $E_{\p}$, and a map $f : E_{\p} \inj K_{v}$ that extends to $\overline{E_{\p}} \simeq \overline{K_v}$. Then, we define $\chi_f $ by taking the following composition 
\[
\chi_{\pi_{\p},f} : \Gal (\overline{K_v}/K_v) \overset{f}\to \Gal (E^{\ab}_{\p}/E_{\p}) \simeq \mathscr{O}^{\times}_{E_{\p}} \times \widehat{\Z} \overset{\pr}\to \mathscr{O}^{\times}_{E_{\p}}.
\]
This does not depend on the choice of the isomorphism $\overline{E_{\p}} \simeq \overline{K_v}$.

\begin{dfn}
\label{4.4}
We denote the $\Gal (\overline{K_{v}}/K_{v})$-module $\mathscr{O}_{E_{\p}}$ arising from a character $\chi : \Gal (\overline{K_{v}}/K_{v}) \to \mathscr{O}_{E_{\p}}^{\times}$ by  $(\mathscr{O}_{E_{\p}},\chi)$. 
\end{dfn}

Next, we define a character $\widetilde{\chi}$ which is required to compute $H^4_{\et} (X_{\overline{K_v}}, \Z_p)$ as a $\Gal (\overline{K_v}/K_v)$-module via $\Z_p$. Let
\[
\chi_{\cyc}: \Gal (\overline{K_v}/K_v) \to \Z^{\times}_{p}
\]
be the $p$-adic cyclotomic character. A free $\Z_p$-module of rank $1$ on which $\Gal (\overline{K_v}/K_v)$ acts via $\chi^{n}_{\cyc}$ is denoted by $\Z_p (n)$. If  $n=0$, then we put $\Z_p (0) := \Z_p$.

\begin{dfn}
\label{4.5}
Let $\iota: E_{\p} \inj K_v$ be an embedding induced by $E \subset K \subset \C$, and $c: E_{\p} \to E_{\p}$ the isomorphism induced from the complex conjugation on $E$.  Then, we have characters $\chi_{\iota},\chi_{\iota \circ c}$ given by the Definition \ref{4.2}. Let $\chi_{\cyc}$ be the $p$-adic cyclotomic character. We define the character $\widetilde{\chi}: \Gal (\overline{K_v}/K_v) \to \mathscr{O}^{\times}_{E_{\p}}$ by
\[
\widetilde{\chi}: = \chi^{-1}_{\iota} \chi_{\iota \circ c} \chi^{-2}_{\cyc}.
\]
\end{dfn}

\begin{rmk}
\label{4.6}
The exponent of $\chi_{\cyc}$ in the definition of $\widetilde{\chi}$ differs from the one used by Kazuhiro Ito because we are considering $H^4$ rather than $H^2$. Since the image of $\CH^2 (X_k) \otimes W$ is a $\Z_p$-lattice, and we want to know the discriminant of $H^4_{\crys}(X_k /W)^{F=p^2}$, the definition must be modified for cubic fourfolds. Thus, we will compute a part which become a lattice via $F=p,p^2,p^3$ to know  its Artin invariant. 
\end{rmk}

In this section, we assumed that $X_k$ has good reduction and the reduction $\mathscr{X}_{\overline{v}}$ becomes supersingular cubic fourfold over $\overline{k(v)}$. To prove Theorem \ref{1.4}, we may replace $K$ by a finite extension. Hence we may assume
\[
A (X_{\C} ) \simeq A (X_K) \simeq A (X_{K_v}),
\]
\[
A (\mathscr{X}_{\overline{v}}) \simeq A (\mathscr{X} \otimes_{\mathscr{O}_{K_v}} k(v)).
\]

Under these assumptions, the action of $\Gal (\overline{K_v}/K_v) $ on $H^4_{\et} (X_{\overline{K_v}},\Z_p)$ is described as follows.

\begin{prop}
\label{4.7}
There is an isomorphism of $\Gal (\overline{K_v}/K_v)$-modules 
\[
H^4_{\et}(X_{\overline{K_v}},\Z_p) \simeq \Z_p (-2)^{\oplus23-[E_{\p}:\Q_p]} \oplus (\mathscr{O}_{E_{\p}},\widetilde{\chi}). 
\]
\end{prop}

\begin{proof}
By the first assumptions of Theorem \ref{1.4}, $\disc A (X_{\C})$ is not divisible by $p$. Hence we have
\[
H^4_{\et}(X_{\overline{K_v}},\Z_p (2)) \simeq (A (X_{\C}) \otimes_{\Z} \Z_p) \oplus (T(X_{\C}) \otimes_{\Z} \Z_p)
\]

By the assumption $A (X_{\C} ) \simeq A (X_K) \simeq A (X_{K_v}) $, the action of $\Gal (\overline{K_v}/K_v)$ on $A (X_{\C}) \otimes_{\Z} \Z_p \simeq A(X_{K_{v}}) \otimes_{\Z} \Z_p$ is trivial. So, it is enough to show that there is an isomorphism of $\Gal (\overline{K_v}/K_v)$-modules
\[ 
T(X_{\C}) \otimes_{\Z} \Z_p \simeq \Z^{\oplus 23-\rank_{\Z} A (X_{\C}) -[E_{\p}:\Q_p]}_{p} \oplus (\mathscr{O}_{E_{\p}}, \chi^{-1}_{\iota} \chi_{\iota \circ c}).
\]

We use the same method as in \cite[Proposition 5.2]{ItoK}. Let $\rho: \Gal (K^{\ab}/K) \to G(\A_f)$ be the continuous homomorphism from the main theorem of complex multiplication for cubic fourfolds (Theorem \ref{2.9}). Since $\Gal (K^{\ab}_v/K_v)$ is compact, the composite
\[
\Gal (K^{\ab}_v/K_v) \to \Gal (K^{\ab}/K) \overset{\rho}{\to} G(\A_f) \to (E \otimes_{\Q} \Q_p)^{\times}
\]
factors through a continuous homomorphism
\[
\rho_p : \Gal (K^{\ab}_v /K_v) \to (\mathscr{O}_{E} \otimes_{\Z} \Z_p)^{\times}.
\]
By the second assumption of Theorem \ref{1.4}, we have
\[
\End_{\Hdg} (T(X_{\C})) \otimes_{\Z} \Z_{(p)} \simeq \mathscr{O}_{E} \otimes_{\Z} \Z_{(p)}.
\]
Using Lemma \ref{4.2}, we see that $T(X_{\C}) \otimes_{\Z} \Z_p$ is a free $\mathscr{O}_{E} \otimes_{\Z} \Z_p$-module of rank $1$ and there is a decomposition 
\[
\mathscr{O}_{E} \otimes_{\Z} \Z_p \simeq \prod_{\p' |p} \mathscr{O}_{E_{\p'}},
\]
where $\p'$ is a finite place of $E$ above $p$. Hence we have an isomorphism of $\Gal (\overline{K_v}/K_{v})$-modules
\[
T(X_{\C}) \otimes_{\Z} \Z_p \simeq \prod_{\p' |p} \mathscr{O}_{E_{\p'}},
\]
where the action of $\Gal (\overline{K_v}/K_v)$ is given by $\pr_{\p'} \circ \rho_p$ as
\[
\Gal (K^{\ab}_v /K_v) \overset{\rho_p}\to
(\mathscr{O}_{E} \otimes_{\Z} \Z_p)^{\times} \simeq \prod_{\p' |p} \mathscr{O}_{E_{\p'}}^{\times} \overset{\pr_{\p'}}\to \mathscr{O}_{E_{\p'}}^{\times}.
\] 
Therefore, it suffices to prove that 
\[
g \mapsto
\begin{cases}
1 & (\p' \not= \p) \\
\chi^{-1}_{\iota}(g) \chi_{\iota \circ c}(g) &(\p' = \p).
\end{cases} 
\]

It suffices to prove these equalities when we restrict them to $K_v^{\times}$ along the local Artin reciprocity map
\[
\Art_{K_v} : K_{v}^{\times} \inj \Gal (K^{\ab}_v /K_v)
\]
since $\Art_{K_v}(K_{v}^{\times}) \subset \Gal (K^{\ab}_v /K_v)$ is dense.

Take an element $x \in K_{v}^{\times}$. By the commutative diagram of Theorem \ref{2.9}, there exists $\eta \in G(\Q)^{\times} \subset E^{\times}$ such that
\[
\rho(\Art_{K_v}(x)) = (1,\dots,1,c(\Nm_{K_v/E_{\p}}(x))\Nm_{K_v/E_{\p}}(x)^{-1},1,\dots,1 ) \eta
\]
because $\q$ does not split in $E$.

Now there is the isomorphism $A(\mathscr{X}_{\overline{v}}) \simeq A(\mathscr{X} \otimes_{\mathscr{O}_{K_v}} k(v))$ and $\mathscr{X}_{\overline{v}}$ is supersingular, the action of $\Gal(\overline{K_v}/K_v)$ on $T(X_{\C}) \otimes_{\Z} \Q_l$ is trivial for all $l \not= p$. It follows that $\eta = 1$. Therefore, we have
\[
(\pr_{\p'} \circ \rho_p) (\Art_{K_v}(x)) =
\begin{cases}
1 & (\p' \not= \p) \\
c(\Nm_{K_v/E_{\p}}(x))\Nm_{K_v/E_{\p}}(x)^{-1} &(\p' = \p).
\end{cases} 
\]
Recall that we fixed the uniformizer $\pi_{\p}$ of $E_{\p}$. It gives an isomorphism
\[
\mathscr{O}_{E_{\p}}^{\times} \times \Z \simeq E_{\p}^{\times};\ \ \ \  (x,n) \mapsto x \pi_{\p}^n.
\]
The composite of the inverse map of this isomorphism with the projection is denoted by $\chi': E_{\p}^{\times} \to \mathscr{O}_{E_{\p}}^{\times}$. Then, we have
\[
\begin{cases}
\chi_{\iota}(\Art_{K_v} (x)) = \chi'(\Nm_{K_v/E_{\p}}(x)) \\
\chi_{\iota \circ c}(\Art_{K_v} (x)) = \chi'(c\Nm_{K_v/E_{\p}}(x)) .
\end{cases} 
\]
For every $z_1,z_2 \in E_{\p}^{\times}$ which have the same valuation, we have
\[
\chi'(z_1) \chi'(z_2)^{-1} = \chi'(z_1 z_2^{-1}) =z_1 z_2^{-1}. 
\]
Hence we have
\begin{align*}
\chi_{\iota}^{-1}(\Art_{K_v}(x)) \chi_{\iota \circ c}(\Art_{K_v} (x)) &= \chi'(c (\Nm_{K_v /E_{\p}}(x) ))\chi'(\Nm_{K_v/E_{\p}}(x))^{-1} \\
&= c (\Nm_{K_v /E_{\p}}(x) ) \Nm_{K_v/E_{\p}}(x)^{-1}.
\end{align*}
\end{proof}

\subsection{The Breuil-Kisin module for cubic fourfold with CM}
\label{sec4.3} 
In this section, we will compute $(\mathscr{O}_{E_{\p}}, \chi^{-1}_{\iota} \chi_{\iota \circ c})$ as $\Gal (\overline{K_v}/K_v)$-modules (cf. Proposition \ref{4.7}) to obtain $H^4_{\crys}(\mathscr{X}_{\overline{v}}/W)$ as $F$-crystal, where $\mathscr{X}_{\overline{v}}$ is cubic fourfold over $\overline{k(v)}$ and $W:=W(\overline{k(v)})$ $ (= W(\overline{\F_p}))$.

The main of this section is to prove Theorem \ref{4.12} which can describe the $F$-crystal as $W$-module.  
We use the consequences of Breuil-Kisin module which can characterize $\Z_p$-lattice in crystalline representation as Theorem \ref{4.8}. By Theorem \ref{4.9}, it is enough to know the structure of integral $p$-adic $\mathrm{\etale}$ cohomology. In particular, the Breuil-Kisin modules defined from $\Gal(\overline{K_v}/K_v)$-modules defined from varieties with CM, can be computed as Proposition \ref{4.10} and Proposition \ref{4.11}.

Fix a uniformizer $\pi_{\p}$ of $E_{\p}$. Let $K^{\ur}_{v}$ be a maximal unramified extension of $K_v$ in $\overline{K_v}$, and $L$ be a completion of $K^{\ur}_{v}$, i.e. $L:= \widehat{K^{\ur}_{v}}$. Now we identify $W$ with the $p$-adic completion of the maximal unramified extension of $\Z_p$ in $\overline{K_v}$. Then $L$ is a finite totally ramified extension of $\Frac (W)$ since using an embedding $W \inj L$. 

Let $E_{\p,0} \subset E_{\p}$ be  the maximal unramified extension of $\Q_p$ in $E_{\p}$. To compute $\mathscr{O}_{E_{\p}} \otimes_{\Z_p} W$ as an $F$-crystal, we decompose it into $\mathscr{O}_{E_{\p,0}}$, which is the subring of $\mathscr{O}_{E_{\p}}$ and constructed as follows: Let $\Emb (E_{\p,0})$ be a set of embeddings from $E_{\p,0}$ to $\Frac (W) \subset L$ and $\varphi$ be a Frobenius on $\Frac (W)$. Take an element $\iota_{0} \in \Emb (E_{\p,0})$ and consider an induced extension
\[
E_{\p,0} \inj E_{\p} \overset{\iota}\inj K_{v} \inj L .
\]
Then, we have
\[
\Emb (E_{\p,0}) = \{ \varphi^i \circ \iota_{0}\ |\ 0 \leq i \leq d-1\},
\]
where $d= [E_{\p,0} :\Q_p]$. Note that, since $E_{\p}/F_{\q}$ is quadratic unramified extension, $d$ is even. 

We introduce some consequences of Breuil-Kisin module. See also \cite[Section 2]{Kisin}. Fix a uniformizer $\pi_{L} \in L$, and a system $\{\pi_{L}^{1/{p^n}}\}_{n \geq 0} \subset \overline{L}$ of $p^n$-th roots of $\pi_{L}$ such that 
\[
(\pi_{L}^{1/p^{n+1}})^p = \pi_{L}^{1/p^n}
\]
for $n \geq 0$. Let $E(u) \in W(u)$ be the monic minimal polynomial of $\pi_L$ over $\Frac (W)$, i.e. it is the (monic) Eisenstein polynomial for $\pi_L$. We define a Frobenius $\varphi$ on
\[
\mathfrak{S} := W[[u]]
\]
which acts on $W$ as the canonical Frobenius and sends $u$ to $u^p$. A \emph{Breuil-Kisin} module (over $\mathcal{O}_{L}$ with respect to $\{\pi_{L}^{1/p^n}\}_{n \geq 0}$) is a free $\mathfrak{S}$-module $\mathfrak{M}$ of finite rank equipped with an isomorphism of $\mathfrak{S}$-modules 
\[
\varphi: (\varphi^{*} \mathfrak{M}) [1/E(u)] \simeq \mathfrak{M} [1/E(u)].
\]

Let $\mathscr{O_{\mathcal{E}}}$ be the $p$-adic completion of $\mathfrak{S}[1/u]$. This is a Cohen ring whose residue field is $\overline{\F_{p}}((u))$. We fix a separable closure $\overline{\F_{p}}((u))^{\sep}$ of $\overline{\F_{p}}((u))$. Let $\mathscr{O}_{\mathcal{E}_{\ur}}$ be the strict henselization of $\mathscr{O}_{\mathcal{E}}$ with respect to $\overline{\F_{p}}((u))^{\sep}$, and $\mathscr{O}_{\widehat{\mathcal{E}_{\ur}}}$  the $p$-adic completion of $\mathscr{O}_{\mathcal{E}_{\ur}}$. The Frobenius $\varphi$ on $\mathfrak{S}$ extends naturally to $\mathscr{O}_{\widehat{\mathcal{E}_{\ur}}}$.

We put 
\[
L_{\infty} := \bigcup_{n \geq 0} L(\pi_{L}^{1/p^n})
\]
By \cite[Corollaire 3.2.3]{Wintenberger}, there is an isomorphism 
\[
\Gal (\overline{L}/L_{\infty}) \simeq \Gal (\overline{\F_p} ((u))^{\sep} / \overline{\F_{p}}((u))).
\]
We have an action of $\Gal(\overline{L}/L_{\infty})$ on $\mathscr{O}_{\widehat{\mathcal{E}_{\ur}}}$ via this isomorphism. 

It is known that the crystalline Galois representation with $\Z_p$-lattice is classified by Breuil-Kisin modules. More precisely, the following are known.

\begin{thm}{\cite{Kisin}}
\label{4.8}

There exists a covariant fully faithful tensor functor 

\begin{center}
\{Category of $\Z_p$-lattice in crystalline $\Gal (\overline{L}/L)$-representation\} \\
$\to$ \{Category of Breuil-Kisin modules\}
\end{center}
\[
\Lambda \mapsto \mathfrak{M}(\Lambda)
\]

The Breuil-Kisin module $\mathfrak{M}(\Lambda)$ is characterized by the existance of an isomorphism of $\mathscr{O}_{\widehat{\mathcal{E_{\ur}}}}$-modules
\[
\Lambda \otimes_{\Z_p} \mathscr{O}_{\widehat{\mathcal{E_{\ur}}}} \simeq \mathfrak{M}(\Lambda) \otimes_{\mathfrak{S}} \mathscr{O}_{\widehat{\mathcal{E_{\ur}}}}
\]
compatible with Frobenius morphism and the action of $\Gal(\overline{L}/L_{\infty})$, where

the Frobenius morphism on the LHS (resp. RHS) is given by $1 \otimes \varphi$ (resp. $\varphi \otimes \varphi$), 

and the action of an element $g \in \Gal (\overline{L}/L_{\infty})$ on the LHS (resp. RHS) is given by $g \otimes g$ (resp $1 \otimes g$).
\end{thm}
\begin{proof}
See \cite{Kisin} and \cite{ItoK}. 
\end{proof}

Moreover, we can obtain $H^4_{\crys}(\mathscr{X}_{\bar{v}}/W) $ as $F$-crystal from integral comparison theorem.

\begin{thm}{\cite{BMSa}}
\label{4.9} 
There is an isomorphism of $F$-crystals
\[
H^4_{\crys}(\mathscr{X}_{\overline{v}}/W) \simeq \varphi^{*} (\mathfrak{M} (H^4_{\et}(X_{\overline{K_v}},\Z_p)) /u\mathfrak{M}(H^4_{\et}(X_{\overline{K_v}},\Z_p)))
\]
\end{thm}
\begin{proof}
See \cite[Theorem 14.6]{BMS} and \cite[Theorem 1.2]{BMSa}, \cite[Theorem 2.4]{ChiarellottoLazdaLiedtke}. 
\end{proof}

So, by Theorem \ref{4.8} and Theorem \ref{4.9}, it is sufficient to compute the associated Galois module $\mathfrak{M}(H^4_{\et}(X_{\overline{K_v}},\Z_p))$ to obtain its Artin invariant. Using Proposition \ref{4.7}, we can describe it as a Lubin-Tate character $(\mathscr{O}_{E_{\p}},\widetilde{\chi})$. The Breuil-Kisin module that comes from this character can be calculated as follows.

Next, we give an explicit description of the Breuil-Kisin module $\mathfrak{M}(H^4_{\et}(X_{\overline{K_v}},\Z_p))$. By Proposition \ref{4.7}, we have an isomorphism 
\[
H^4_{\et}(X_{\overline{K_v}},\Z_p) \simeq \Z_p (-2)^{\oplus 23-[E_{\p}:\Q_p]} \oplus (\mathscr{O}_{E_{\p}},\widetilde{\chi}).
\]
So there is an isomorphism of Breuil-Kisin modules
\[\mathfrak{M}(
H^4_{\et}(X_{\overline{K_v}},\Z_p)) \simeq \mathfrak{M}(\Z_p (-2))^{\oplus 23-[E_{\p}:\Q_p]} \oplus \mathfrak{M}((\mathscr{O}_{E_{\p}},\widetilde{\chi})).
\]

The Breuil-Kisin module $\mathfrak{M}(\Z_p(-2))$ is described as follows.  It is identified with
$\mathfrak{S}$ as an $\mathfrak{S}$-module, and the Frobenius $\varphi$ is defined by $\varphi(1) = p^2 E(u)/E(0)$. Thus, to describe $\mathfrak{M}(
H^4_{\et}(X_{\overline{K_v}},\Z_p))$, it is enough to compute $\mathfrak{M}((\mathscr{O}_{E_{\p}},\widetilde{\chi}))$. 

We put 
\[
W_i := \mathscr{O}_{E_{\p}} \otimes_{R_i} W,
\]
where $R_i$ is the subring of $W$ by the embedding from  $\mathscr{O}_{E_{\p,0}}$ through $\varphi^i \circ \iota$. 

Then we have a decomposition
\[
\mathscr{O}_{E_{\p}} \otimes_{\Z_p} W \simeq \prod_{0 \leq i \leq d-1} W_i,
\]
and
\[
\mathscr{O}_{E_{\p}} \otimes_{\Z_p} \mathfrak{S} \simeq \prod_{0 \leq i \leq d-1} \mathfrak{S}_i,
\]
where $\mathfrak{S}_i := W_i [[u]]$.

Let 
\[
E_0 (u) \in W_{0}[u]
\]
be the monic minimal polynomial of $\pi_L$ over $\Frac(W_0)$. Now we consider the case of $i =0,d/2$. $W_0$ is embedded into $L$ via the map
\[
W_0 = \mathscr{O}_{E_{\p}} \otimes_{R_0} W \to L, \ \ \ \ x \otimes y \mapsto \iota (x) y.
\]

By Proposition \ref{4.3}, the involution $c$ is non-trivial on $E_{\p,0}$. So the embedding $\iota \circ c : E_{\p} \inj L$ induces the embedding $\varphi^{d/2} \circ \iota_{0} \in \Emb(E_{\p,0})$. Let 
\[
E_{d/2} (u) \in W_{d/2}[u]
\]
be the monic minimal polynomial of $\pi_{L}$ over $\Frac (W_{d/2})$, where we embed $W_{d/2}$ into $L$ by the embedding 
\[
W_{d/2} = \mathscr{O}_{E_{\p}} \otimes_{R_{d/2}} W \to L, \ \ \ \ x \otimes y \mapsto (\iota \circ c) (x) y.
\]

Recall that the character  $\widetilde{\chi} : \Gal(\overline{K_{v}}/K_v) \to \mathscr{O}_{E_{\p}}^{\times}$ is defined by
\[
\widetilde{\chi} = \chi_{\iota}^{-1} \chi_{\iota \circ c} \chi_{\cyc}^{-2}.
\]
To know the inverse part of this character, let $(\mathscr{O}_{E_{\p}},\chi)^{\vee}$ be a $\Gal (\overline{K_v}/K_v)$-module
\[
\mathscr{O}_{E_{\p}}^{\vee} := \Hom_{\Z_p} (\mathscr{O}_{E_{\p},\Z_p})
\]
induced by $(\mathscr{O}_{E_{\p}},\chi)$. We show the inverse part of $\widetilde{\chi}$ can be computed from the dual $\mathscr{O}_{E_{\p}}^{\vee}$.

For an element $\gamma \in \mathscr{O}_{E_{\p}} \otimes_{\Z_p} \mathfrak{S}$, we define a Frobenius semi-linear map $\varphi_{\gamma}$ on $\mathscr{O}_{E_{\p}} \otimes_{\Z_p} \mathfrak{S}$ which commutes with the $\mathscr{O}_{E_{\p}}$-action by $\varphi_{\gamma} (1) = \gamma$. We denote the $\mathfrak{S}$-module $\mathscr{O}_{E_{\p}} \otimes_{\Z_p} \mathfrak{S}$ equipped with the Frobenius semi-linear map $\varphi_{\gamma}$ by
\[
(\mathscr{O}_{E_{\p}} \otimes_{\Z_p} \mathfrak{S}, \gamma).
\]

\begin{prop}{\cite[Proposition 2.2.1]{AGHM}}
\label{4.10}

(1)\ The Breuil-Kisin module 
\[
\mathfrak{M} ((\mathscr{O}_{E_{\p}},\chi_{\iota})^{\vee})
\]
is $\mathscr{O}_{E_{\p}}$-equivalently isomorphic to Breuil-Kisin module $(\mathscr{O}_{E_{\p}} \otimes_{\Z_p} \mathfrak{S},\be_{\iota})$, where the element
\[
\be_{\iota} \in \mathscr{O}_{E_{\p}}\otimes_{\Z_p} \mathfrak{S} \simeq \prod_{0 \leq i \leq d-1} \mathfrak{S}_{i}
\]
is specified as follows; the $i$-th component $\be_{\iota, i}$ of $\be_{\iota}$ is
\[
\be_{\iota, i} = 
\begin{cases}
\pi_{\p} E_{0}(u) /E_{0} (0)\ \ \ \ &(i=0) \\
1 &(i \not= 0).
\end{cases}
\]

(2)\ The Breuil-Kisin module 
\[
\mathfrak{M} ((\mathscr{O}_{E_{\p}},\chi_{\iota \circ c})^{\vee})
\]
is $\mathscr{O}_{E_{\p}}$-equivalently isomorphic to Breuil-Kisin module $(\mathscr{O}_{E_{\p}} \otimes_{\Z_p} \mathfrak{S},\be_{\iota \circ c})$, where the element
\[
\be_{\iota} \in \mathscr{O}_{E_{\p}}\otimes_{\Z_p} \mathfrak{S} \simeq \prod_{0 \leq i \leq d-1} \mathfrak{S}_{i}
\]
is specified as follow; the $i$-th component $\be_{\iota \circ c, i}$ of $\be_{\iota \circ c}$ is
\[
\be_{\iota \circ c, i} = 
\begin{cases}
\pi_{\p} E_{d/2}(u) /E_{d/2} (0)\ \ \ \ &(i=d/2) \\
1 &(i \not= d/2)
\end{cases}
\]
\end{prop}
\begin{proof}
See \cite[Proposition 2.2.1]{AGHM} and \cite[Proposition 6.3]{ItoK}.
\end{proof}

\begin{prop}{cf. \cite[Proposition 6.4]{ItoK}} 
\label{4.11}
The Breuil-Kisin module $\mathfrak{M}((\mathscr{O}_{E_{\p}}, \widetilde{\chi}))$ is $\mathscr{O}_{E_{\p}}$-equivariantly isomorphic to the Breuil-Kisin module $(\mathscr{O}_{E_{\p}}\otimes_{\Z_p} \mathfrak{S},\widetilde{\be})$, where the element 
\[
\widetilde{\be} \in \mathscr{O}_{E_{\p}} \otimes_{\Z_p} \mathfrak{S} \simeq \prod_{0 \leq i \leq d-1} \mathfrak{S}_i
\]
is specified as follows; the $i$-th component $\widetilde{\be}_i$ of $\widetilde{\be}$ is 
\[
\widetilde{\be}_i = 
\begin{cases}
p^2 (E(u)/E(0)) (\pi_{\p} E_0 (u)/E_0 (0)) & (i=0) \\
p^2 (E(u)/E(0)) (\pi_{\p} E_{d/2} (u)/E_{d/2} (0))^{-1} & (i=d/2) \\
p^2 (E(u)/E(0))  & (i \not=0 ,d/2) \\
\end{cases}
\]
\end{prop}
\begin{proof}
By an isomorphism of $\Gal(\overline{K_v}/K_v)$-modules
\[
(\mathscr{O}_{E_{\p}}, \chi^{-1}_{\iota})\left[\frac{1}{p}\right] \simeq (\mathscr{O}_{E_{\p}}, \chi^{-1}_{\iota})^{\vee}\left[\frac{1}{p}\right]
\]
induced by the trace map $\Tr_{E_{\p}/\Q_p}:E_{\p} \to \Q_p$, the $\mathscr{O}_{E_{\p}}$-lattice $\mathscr{O}_{E_{\p}}^{\vee}$ on the RHS is translated to the $\mathscr{O}_{E_{\p}}$-lattice $\pi_{\p}^{-d} \mathscr{O}_{E_{\p}} \subset (\mathscr{O}_{E_{\p}}, \chi^{-1}_{\iota})\left[\frac{1}{p}\right])$ for an integer $d>0$. It follows that there exists an isomorphism of $\Gal(\overline{K_v}/K_v)$-modules
\[
(\mathscr{O}_{E_{\p}}, \chi^{-1}_{\iota}) \simeq (\mathscr{O}_{E_{\p}}, \chi^{-1}_{\iota})^{\vee}
\]
Thus, there is an isomorphism of $\Gal(\overline{L}/L)$-modules
\[
(\mathscr{O}_{E_{\p}}, \widetilde{\chi}) \simeq (\mathscr{O}_{E_{\p}}, \chi_{\iota \circ c} \chi_{\cyc}^{-2}) \otimes_{\mathscr{O}_{E_{\p}}} (\mathscr{O}_{E_{\p}}, \chi_{\iota})^{\vee}.
\]
Taking the functor $\mathfrak{M}(-)$, we have
\[
\mathfrak{M}((\mathscr{O}_{E_{\p}}, \widetilde{\chi})) \simeq \mathfrak{M}((\mathscr{O}_{E_{\p}}, \chi_{\iota \circ c} \chi_{\cyc}^{-2}) \otimes_{\mathscr{O}_{E_{\p}}} (\mathscr{O}_{E_{\p}}, \chi_{\iota})^{\vee}).
\]

Since $\mathfrak{M}((\mathscr{O}_{E_{\p}},\chi_{\iota})^{\vee})$ is a free $\mathscr{O}_{E_{\p}} \otimes_{\Z_p} \mathfrak{S}$-module of rank $1$ by Proposition \ref{4.10} (1), $\mathfrak{M}((\mathscr{O}_{E_{\p}}, \chi_{\iota \circ c} \chi_{\cyc}^{-2})) \otimes_{\mathscr{O}_{E_{\p}} \otimes_{\Z_p} \mathfrak{S}} \mathfrak{M}((\mathscr{O}_{E_{\p}}, \chi_{\iota})^{\vee})$ is a finite rank $\mathfrak{S}$-module, i.e., it is Breuil-Kisin module. In particular, using Theorem \ref{4.8}, we have an isomorphism 
\[
\mathfrak{M}((\mathscr{O}_{E_{\p}}, \widetilde{\chi})) \simeq \mathfrak{M}((\mathscr{O}_{E_{\p}}, \chi_{\iota \circ c} \chi_{\cyc}^{-2})) \otimes_{\mathscr{O}_{E_{\p}} \otimes_{\Z_p} \mathfrak{S}} \mathfrak{M}((\mathscr{O}_{E_{\p}}, \chi_{\iota})^{\vee}).
\]

Using an isomorphism $\mathscr{O}_{E_{\p}} \simeq \mathscr{O}_{E_{\p}}^{\vee}$ as above, Proposition \ref{4.10} (2) and the fact that the functor $\mathfrak{M}(-)$ is a tensor functor, we see that the Breuil-Kisin module 
\[
\mathfrak{M}((\mathscr{O}_{E_{\p}},\chi_{\iota \circ c}\chi_{\cyc}^{-2}))
\]
is $\mathscr{O}_{E_{\p}}$-equivariantly isomorphic to the Breuil-Kisin module $(\mathscr{O}_{E_{\p}} \otimes_{\Z_p} \mathfrak{S}, \be')$, where the element 
\[
\be' \in \mathscr{O}_{E_{\p}} \otimes_{\Z_p} \mathfrak{S} \simeq \prod_{0 \leq i \leq d-1} \mathfrak{S}_i
\]
is specified as follows; the $i$-th component $\be'_i$ of $\be$ is
\[
\be'_i = 
\begin{cases}
p^2 (E(u)/E(0)) (\pi_{\p} E_{d/2}(u)/E_{d/2}(0))^{-1} & (i=d/2) \\
p^2 (E(u)/E(0)) & (i \not= d/2) 
\end{cases}
\]
By using this description, the assertion follows from the isomorphism 
\[
\mathfrak{M}((\mathscr{O}_{E_{\p}},\widetilde{\chi})) \simeq \mathfrak{M}((\mathscr{O}_{E_{\p}}, \chi_{\iota \circ c} \chi_{\cyc}^{-2})) \otimes_{\mathscr{O}_{E_{\p}} \otimes_{\Z_p} \mathfrak{S}} \mathfrak{M}((\mathscr{O}_{E_{\p}}, \chi_{\iota})^{\vee}).
\]
and Theorem \ref{4.8}.
\end{proof}

We prove the following Theorem \ref{4.12}, which provides a detailed description of $H^4_{\crys}(\mathscr{X}_{\bar{v}}/W)$ as an $F$-crystal.

\begin{thm}{cf. \cite[Theorem 6.1]{ItoK}}
\label{4.12}
The $F$-crystal $H^4_{\crys}(\mathscr{X}_{\overline{v}}/W)$ is isomorphic to the $F$-crystal
\[
W(-2)^{\oplus 23-[E_{\p}:\Q_{p}]} \oplus (\mathscr{O}_{E_{\p}} \otimes_{\Z_p} W,\be)
\]
Here, the $F$-crystal $(\mathscr{O}_{E_{\p}} \otimes_{\Z_p} W,\be)$ is a $W$-module $\mathscr{O}_{E_{\p}} \otimes_{\Z_p} W$ equipped with an $\mathscr{O}_{E_{\p}}$-equivariant Frobenius $\varphi$ given by $\varphi(1) = \be$, and the element
\[
\be \in \mathscr{O}_{E_{\p}} \otimes_{\Z_p} W \simeq \prod_{0 \leq i \leq d-1}W_{i}
\]
is specified as follows; the $i$-th component $\be_i$ of $\be$ is
\[
\be_{i} = 
\begin{cases}
p^2 \pi_{\p} &(i=1) \\
p^2 \pi_{\p}^{-1} &(i=d') \\
p^2 &(i \not= 1,d')
\end{cases}
\]
where we put $d' = d/2+1$ if $d\not= 2$, and $d'=0$ if $d=2$.
\end{thm}

\begin{proof}
It holds from Theorem \ref{4.9}, Proposition \ref{4.11}. 
\end{proof}

\section{Proof of Theorem \ref{1.4}}
\label{sec5}
\begin{proof}
We assume $p \geq 7$. The cycle map 
\[
\cl^{2}_{\crys}: \CH^2 (\mathscr{X}_{\overline{v}}) \otimes_{\Z} W \to H^4_{\crys}(\mathscr{X}_{\overline{v}}/W)
\]
is injective and preserves pairing. Since Artin invariant $a$ satisfies $\disc \CH^2 (\mathscr{X}_{\overline{v}}) = \pm p^{2a}$, it is equal to $[k(\q) : \F_{\p}] = [k(\p) : \F_{\p}]/2 = d/2$. (Note that the extension $E_{\p}/F_{\q}$ is unramified.)

By theorem \ref{4.12}, we have an isomorphism of $F$-crystals
\[
H^{4}_{\crys} (\mathscr{X}_{\overline{v}}/W) \simeq W(-2)^{\oplus 23-[E_{\p}:\Q_p]} \oplus (\mathscr{O}_{E_{\p}} \otimes_{\Z_p} W, \be)
\]
where $\be \in \mathscr{O}_{E_{\p}} \otimes_{\Z_p} W$ is taken as in Theorem \ref{4.12}. By \cite[Proposition (5.2.2)]{Yang}, we have
\[
A(\mathscr{X}_{\overline{v}}) \otimes_{\Z} \Z_p \simeq H^{4}_{\crys}(\mathscr{X}_{\overline{v}}/W)^{\varphi = p^2}. 
\]
Note that this isomorphism is of $\Z_p$-modules. Hence the length of cokernel of the following injection
\[
H^{4}_{\crys}(\mathscr{X}_{\overline{v}}/W)^{\varphi = p^2} \otimes_{\Z_p} W \to H^{4}_{\crys}(\mathscr{X}_{\overline{v}}/W)
\]
as $W$-modules is equal to Artin invariant $a$. 

Recall that $W(-2)$ is the $F$-crystal of rank $1$ whose underlying $W$-module is $W$ and Frobenius $\varphi$ is given by $\varphi(1) = p^2$. Hence $1 \in W$ is $\Z_p$-basis of $W(-2)^{\varphi=p^2}$, and the canonical map 
\[
W(-2)^{\varphi = p^2} \otimes_{\Z_p} W \to W
\]
is an isomorphism. Therefore, the Artin invariant $a$ is equal to the length of the cokernel of the following injection
\[
(\mathscr{O}_{E_{\p}} \otimes_{\Z_p} W,\be)^{\varphi = p^2} \otimes_{\Z_p} W \to (\mathscr{O}_{E_{\p}} \otimes_{\Z_p} W,\be) 
\]
as a $W$-module.

We shall compute a $\Z_p$-basis of $(\mathscr{O}_{E_{\p}} \otimes_{\Z_p} W,\be)^{\varphi= p^2}$. An element $\displaystyle x \in \mathscr{O}_{E_{\p}} \otimes_{\Z_p} W \simeq \prod_{0 \leq i \leq d-1} W_i$ satisfies $\varphi (x)= p^2 x$ if and only if the $i$-th components $x_i \in W_i\ (1 \leq i \leq d-1)$ satisfy
\[
\begin{cases}
\pi_{\p} \varphi (x_0) =x_1 & \\
\pi_{\p}^{-1} \varphi (x_{d/2}) =x_{d/2+1} & \\
 \varphi (x_i) =x_{i+1} & (i \not= 0,d/2)
\end{cases}
\]
where we put $x_d := x_0$. From these equations, we have $\varphi^d (x_{d/2+1}) = x_{d/2+1}$ and $x_{d/2+1} \in \mathscr{O}_{E_{\p}} \subset W_{d/2+1}$ where we put $W_d := W_0$. Conversely, if $x_{d/2+1}$ is an element of $\mathscr{O}_{E_{\p}} \subset W_{d/2+1}$, then for the elements $x_i \in W_i \ (0 \leq i \leq d-1)$ determined by the above equations, the element $x= (x_0, \dots,x_{d-1})$ is in $(\mathscr{O}_{E_{\p}} \otimes_{\Z_p} W, \be)^{\varphi = p^2}$. Thus, the following composite 
\[
\mathscr{O}_{E_{\p}} \otimes_{\Z_p} W \simeq \prod_{0 \leq i \leq d-1} W_i \to W_{d/2 +1}
\]
where the second map is the projection, maps $(\mathscr{O}_{E_{\p}} \otimes_{\Z_p} W, \be)^{\varphi =p^2}$ isomorphically onto $\mathscr{O}_{E_{\p}}$. We denote by $f$ induced isomorphism of $\Z_p$-modules $(\mathscr{O}_{E_{\p}} \otimes_{\Z_p} W,\be)^{\varphi =p^2} \simeq \mathscr{O}_{E_{\p}}$. The following composite 
\[
\mathscr{O}_{E_{\p}} \otimes_{\Z_p} W \overset{f^{-1} \otimes \id_W }\to (\mathscr{O}_{E_{\p}} \otimes_{\Z_p} W,\be)^{\varphi =p^2} \otimes_{\Z_p} W \to \mathscr{O}_{E_{\p}} \otimes_{\Z_p} W 
\]
is equal to the $W$-linear map 
\[
g_{\pi}: \mathscr{O}_{E_{\p}} \otimes_{\Z_p} W \to \mathscr{O}_{E_{\p}} \otimes_{\Z_p} W; x \mapsto \pi x .
\] 
Here $\displaystyle \pi \in \mathscr{O}_{E_{\p}} \otimes_{\Z_p} W \simeq \prod_{0 \leq i \leq d-1} W_i$ is the element whose $i$-th component $\pi_i$ is 
\[
\pi_{i} =
\begin{cases}
1 & (i=0) \\
\pi_{\p} & (1 \leq i \leq d/2) \\
1 & (d/2+1 \leq i \leq d-1)
\end{cases}
\]
Therefore, to show Theorem \ref{1.4}, it suffices to show that the length $N$ of the cokernel of $g_{\pi}$ as a $W$-module is equal to $d/2$. The length of the cokernel of the injection 
\[
W_i \to W_i;\ \ \ \  x \mapsto \pi_{\p} x
\]
as a $W$-module is equal to $1$ since the extension $\Frac(W_i)/\Frac(W)$ is totally ramified and $\pi_{\p} \in W_i$ is a uniformizer of $\Frac(W_i)$. From this, we have
\[
N = d/2 =[k(\q):\F_p]
\]

The proof of Theorem \ref{1.4} is complete.
\end{proof}

\section{Application}

\begin{dfn}{\cite[Definition 2.7]{LazaZheng}} \label{6.1}
An automorphism of a smooth cubic fourfold $X$ is called \emph{symplectic}, if the induced automorphism on $F(X)$ is symplectic. Equivalently, the induced action on $H^{3,1}$ is the identity. An automorphism is called \emph{non-symplectic}, if the induced action on $H^{3,1}$ is not identity. 
\end{dfn}

\begin{cor}\label{6.2}
Let $X_{\C}$ be a cubic fourfold defined by a homogeneous polynomial $F$ as
\[
F:= X_0^2 X_1 + X_1^2 X_2 + X_2^2 X_3 + X_3^2 X_4 + X_4^2 X_5 +X_5^3.
\]

Then $X_{\C}$ has a CM field $\Q (\zeta_{32})$, where $\zeta_{32} \in \C$ is the primitive root $32$-th of unity. Therefore, the rank of $\mathscr{X}_{\overline{v}}$ is $23-16 = 7$ if $p \not \equiv -1 \pmod {32}$.

Moreover, if $p \equiv -1 \pmod {32}$ and satisfying the assumption of the second theorem, then the reduction of this is supersingular (i.e., the rank of $\mathscr{X}_{\overline{v}}$ is $23$) and its Artin invariant $\sigma_0$ is $1$. 
\end{cor}
\begin{proof}
First, $X_{\C}$ is smooth. Take a (diagonal) automorphism $f \in \Aut (X_{\C})$ as
\[
f: X \to X ; X_i \mapsto \lam_i X_i \ \ \ \ (i=0,\dots,5),
\]
where $\lam_i \in \C$. By checking the coefficients, we obtain the following
\[
\lam_0^2 \lam_1 = \lam_1^2 \lam_2 = \lam_2^2 \lam_3 =\lam_3^2 \lam_4 =\lam_4^2 \lam_5 = \lam_5^3 =1. 
\]
Therefore, $\lam_i$ satisfy
\[
\lam_0 = \zeta^{-1}_{96},\lam_1 = \zeta_{48} = \zeta_{96}^{2},\lam_2 = \zeta_{24} =\zeta_{96}^{-4},\lam_3 = \zeta_{12} = \zeta_{96}^{8},\lam_4 = \zeta_{6} = \zeta_{96}^{-16}, \lam_5 = \zeta_{3} = \zeta_{96}^{32},
\]
where $\zeta_{n}$ is primitive $n$-root of unity. Thus, we have $f \in \Aut (X_{\C}) \simeq \Z/32\Z$ as
\[
f: (X_0:X_1:X_2:X_3:X_4:X_5) \mapsto (\zeta_{96}^{-1}X_0:\zeta_{96}^{2}X_1:\zeta_{96}^{-4}X_2:\zeta_{96}^{8}X_3:\zeta_{96}^{-16}X_4:\zeta_{96}^{32}X_5)
\]
We can calculate the order of the induced automorphism $f^{*}$ in $H^{3,1}(X_{\C})$ using the method of \cite[lemma 3.2]{Fu}. Equivalently, we have a non-symplectic automorpshim $f$ which induces
\[
f^{*} : H^{3,1}(X_{\C}) \to H^{3,1}(X_{\C}); x \mapsto \zeta_{96}^{21-2j} \cdot x.
\]
Therefore, using the inequality of rank and Betti number, we have
\[
e \varphi (32) = 16e  \leq 23 - \rho
\]
where $\varphi(n)$ is the cyclotomic polynomial of $n$ and $e$ is the dimension of $T(X_{\C})$ as a $\Q(\zeta_{32})$-vector space and $\rho$ is the rank of $X_{\C}$. So $e=1$, and $X_{\C}$ has a CM field $\Q(\zeta_{32})$. 

Let $X_{K}$ be a mod $p$ reduction of $X_{\C}$. If $p \equiv -1 \pmod {32}$, then the finite place $\q$ of $F$ does not split in $E$. Since the extension $E/\Q$ is Galois, the assumption of the second theorem is satisfied. So we have a cubic fourfold $\mathscr{X}_{\overline{v}}$ over $\overline{k(v)} $ which Artin invariant is $\sigma_0 = [k(\q) : \F_p] = 1 $ when $p \equiv -1 \pmod {32}$.
\end{proof}

\begin{rmk} \label{6.3}
In fact, we can find the cycles of $\rank A(X_{\C}) = 7$. Let $H$ be a hypersurface in $\P^5$ defined by $X_1 = 0$. Then, the intersection $F=H=0$ is cubic threefold defined by $G = X_2^2 X_3 + X_3^2 X_4 + X_4^2 X_5 +X_5^3 $. This cubic threefold is also a projective cone over (smooth) cubic surface $S$. It is known that $\dim H^2 (S) = 7$ and 
\[
7=\dim H^{2} (S) \leq \rank A(X_{\C}) \leq \dim H^4 (X) =23 
\]
So it suffices to take the cycle of $H^2 (S)$.
\end{rmk}

Next, we have an another proof of following lemma.

\begin{cor}{cf. \cite[Lemma (5.2.8)]{Yang}}
Let $E$ be an imaginary quadratic extension of $\Q$. For all but finitely many $\q$ which non-split in $E$, there exists a supersingular cubic fourfold $\mathscr{X}_{\overline{v}}$ over $\overline{k(v)}$ with Artin invariant $\sigma_0 = 1$ such that $\Phi : \mathscr{S}_{\overline{v}}^{[2]} \overset{\sim}\to F(\mathscr{X}_{\overline{v}})$ for some supersingular K3 surface $\mathscr{S}_{\overline{v}}$ with Artin invariant $1$.
\end{cor}
\begin{proof}
For every $E$, by \cite[Lemma (5.2.7)]{Yang}, we can take a complex K3 surface $S_{\C}$ with CM field $E$ and a complex cubic fourfold $X_{\C}$ with CM field $E$ such that they satisfy $S_{\C}^{[2]} \simeq F(X_{\C})$. We apply the Theorem \ref{1.4} to $X_{\C}$ and Kazuhiro Ito's theorem \cite[Theorem 1.2]{ItoK} to $S_{\C}$. Then, we can obtain the supersingular cubic fourfold $\mathscr{X}_{\overline{v}}$ over $\overline{k(v)}$ with Artin invariant $\sigma_0 = 1$, and  supersingular K3 surface $\mathscr{S}_{\overline{v}}$ over $\overline{k(v)}$.  We can specialize the isomorphism $S_{\C}^{[2]} \simeq F(X_{\C})$ to $\mathscr{S}_{\overline{v}}^{[2]} \simeq F(\mathscr{X}_{\overline{v}})$. 
\end{proof}

\bibliography{bib}

@article{Shiodakatsura,
  title={ON FERMAT VARIETIES},
  author={Tetsuji Shioda and Toshiyuki Katsura},
  journal={Tohoku Mathematical Journal, Second Series},
  volume={31},
  number={1},
  pages={97-115},
  year={1979},
  doi={10.2748/tmj/1178229881}
}

@incollection {Liedtke,
    AUTHOR = {Liedtke, Christian},
     TITLE = {Lectures on supersingular {K}3 surfaces and the crystalline
              {T}orelli theorem},
 BOOKTITLE = {K3 surfaces and their moduli},
    SERIES = {Progr. Math.},
    VOLUME = {315},
     PAGES = {171--235},
 PUBLISHER = {Birkh\"auser/Springer, [Cham]},
      YEAR = {2016},
      ISBN = {978-3-319-29958-7; 978-3-319-29959-4},
   MRCLASS = {14F30 (14J10 14J28 14M20)},
  MRNUMBER = {3524169},
MRREVIEWER = {Christopher\ David\ Lazda},
       DOI = {10.1007/978-3-319-29959-4\_8},
       URL = {https://doi.org/10.1007/978-3-319-29959-4_8},
}

@article{ItoK,
  author={Ito, Kazuhiro},
  journal={International Mathematics Research Notices}, 
  title={On the Supersingular Reduction of {K3} Surfaces with Complex Multiplication}, 
  year={2020},
  volume={2020},
  number={10},
  pages={7306-7346},
  doi={10.1093/imrn/rny210}}

@article{ItoR, 
      author={Rikuto Ito},
      title={Arithmetic Period Map and Complex Multiplication for Cubic Fourfolds}, 
      year={2025},
      eprint={2512.11355},
      archivePrefix={arXiv},
      primaryClass={math.NT},
      note={available at \url{https://arxiv.org/abs/2512.11355}}, 
}

@article{Yang, 
title = {On irreducible symplectic varieties of $\mathrm{K3}^{[n]}$-type in positive characteristic},
journal = {Advances in Mathematics},
volume = {417},
pages = {108930},
year = {2023},
issn = {0001-8708},
doi = {https://doi.org/10.1016/j.aim.2023.108930},
url = {https://www.sciencedirect.com/science/article/pii/S0001870823000737},
author = {Ziquan Yang},
}

@incollection{Ogus,
    AUTHOR = {Ogus, Arthur},
     TITLE = {Supersingular {$K3$}\ crystals},
 BOOKTITLE = {Journ\'ees de {G}\'eom\'etrie {A}lg\'ebrique de {R}ennes
              ({R}ennes, 1978), {V}ol. {II}},
    SERIES = {Ast\'erisque},
    VOLUME = {64},
     PAGES = {3--86},
 PUBLISHER = {Soc. Math. France, Paris},
      YEAR = {1979},
   MRCLASS = {14J25 (14B10 14F30)},
  MRNUMBER = {563467},
MRREVIEWER = {G.\ Horrocks},
}

@article {Zucker,
    AUTHOR = {Zucker, Steven},
     TITLE = {The {H}odge conjecture for cubic fourfolds},
   JOURNAL = {Compositio Math.},
  FJOURNAL = {Compositio Mathematica},
    VOLUME = {34},
      YEAR = {1977},
    NUMBER = {2},
     PAGES = {199--209},
      ISSN = {0010-437X,1570-5846},
   MRCLASS = {14C30},
  MRNUMBER = {453741},
MRREVIEWER = {C.\ A. M. Peters},
       URL = {http://www.numdam.org/item?id=CM_1977__34_2_199_0},
}

@book{Serrelf, 
    AUTHOR = {Serre, Jean-Pierre},
     TITLE = {Local fields},
    SERIES = {Graduate Texts in Mathematics},
    VOLUME = {67},
      NOTE = {Translated from the French by Marvin Jay Greenberg},
 PUBLISHER = {Springer-Verlag, New York-Berlin},
      YEAR = {1979},
     PAGES = {viii+241},
      ISBN = {0-387-90424-7},
   MRCLASS = {12Bxx},
  MRNUMBER = {554237},
}

@article{Rizov, 
    AUTHOR = {Rizov, Jordan},
     TITLE = {Kuga-{S}atake abelian varieties of {K}3 surfaces in mixed
              characteristic},
   JOURNAL = {J. Reine Angew. Math.},
  FJOURNAL = {Journal f\"ur die Reine und Angewandte Mathematik. [Crelle's
              Journal]},
    VOLUME = {648},
      YEAR = {2010},
     PAGES = {13--67},
      ISSN = {0075-4102,1435-5345},
   MRCLASS = {14K10 (14D22 14G35 14J10 14J28)},
  MRNUMBER = {2774304},
MRREVIEWER = {Lars\ Halvard\ Halle},
       DOI = {10.1515/CRELLE.2010.078},
       URL = {https://doi.org/10.1515/CRELLE.2010.078},
}

@article{Chen,
    AUTHOR = {Chen, Meng},
     TITLE = {Complex multiplication, rationality and mirror symmetry for
              abelian varieties},
   JOURNAL = {J. Geom. Phys.},
  FJOURNAL = {Journal of Geometry and Physics},
    VOLUME = {58},
      YEAR = {2008},
    NUMBER = {5},
     PAGES = {633--653},
      ISSN = {0393-0440,1879-1662},
   MRCLASS = {14K22 (14J32 32J25 32Q15)},
  MRNUMBER = {2419693},
MRREVIEWER = {Noriko\ Yui},
       DOI = {10.1016/j.geomphys.2008.01.001},
       URL = {https://doi.org/10.1016/j.geomphys.2008.01.001},
}

@article{Kats,
    AUTHOR = {Katz, Nicholas M. and Messing, William},
     TITLE = {Some consequences of the {R}iemann hypothesis for varieties
              over finite fields},
   JOURNAL = {Invent. Math.},
  FJOURNAL = {Inventiones Mathematicae},
    VOLUME = {23},
      YEAR = {1974},
     PAGES = {73--77},
      ISSN = {0020-9910,1432-1297},
   MRCLASS = {14G10},
  MRNUMBER = {332791},
MRREVIEWER = {S.\ S.\ Shatz},
       DOI = {10.1007/BF01405203},
       URL = {https://doi.org/10.1007/BF01405203},
}

@article{AGHM, 
    AUTHOR = {Andreatta, Fabrizio and Goren, Eyal Z. and Howard, Benjamin
              and Madapusi Pera, Keerthi},
     TITLE = {Faltings heights of abelian varieties with complex
              multiplication},
   JOURNAL = {Ann. of Math. (2)},
  FJOURNAL = {Annals of Mathematics. Second Series},
    VOLUME = {187},
      YEAR = {2018},
    NUMBER = {2},
     PAGES = {391--531},
      ISSN = {0003-486X,1939-8980},
   MRCLASS = {11G15 (11F41 11G18 14G35 14G40)},
  MRNUMBER = {3744856},
MRREVIEWER = {Martin\ Orr},
       DOI = {10.4007/annals.2018.187.2.3},
       URL = {https://doi.org/10.4007/annals.2018.187.2.3},
}

@article{Kisin,
    AUTHOR = {Kisin, Mark},
     TITLE = {Integral models for {S}himura varieties of abelian type},
   JOURNAL = {J. Amer. Math. Soc.},
  FJOURNAL = {Journal of the American Mathematical Society},
    VOLUME = {23},
      YEAR = {2010},
    NUMBER = {4},
     PAGES = {967--1012},
      ISSN = {0894-0347,1088-6834},
   MRCLASS = {11G18 (14G35)},
  MRNUMBER = {2669706},
MRREVIEWER = {Jeffrey\ D.\ Achter},
       DOI = {10.1090/S0894-0347-10-00667-3},
       URL = {https://doi.org/10.1090/S0894-0347-10-00667-3},
}

@article {BMSa,
    AUTHOR = {Bhatt, B. and Morrow, M. and Scholze, P.},
     TITLE = {Integral {$p$}-adic {H}odge theory---announcement},
   JOURNAL = {Math. Res. Lett.},
  FJOURNAL = {Mathematical Research Letters},
    VOLUME = {22},
      YEAR = {2015},
    NUMBER = {6},
     PAGES = {1601--1612},
      ISSN = {1073-2780,1945-001X},
   MRCLASS = {14F20 (11G25 14F30 14F40 14G20 14G22)},
  MRNUMBER = {3507252},
MRREVIEWER = {Christopher\ David\ Lazda},
       DOI = {10.4310/MRL.2015.v22.n6.a3},
       URL = {https://doi.org/10.4310/MRL.2015.v22.n6.a3},
}

@article {BMS,
    AUTHOR = {Bhatt, Bhargav and Morrow, Matthew and Scholze, Peter},
     TITLE = {Integral {$p$}-adic {H}odge theory},
   JOURNAL = {Publ. Math. Inst. Hautes \'Etudes Sci.},
  FJOURNAL = {Publications Math\'ematiques. Institut de Hautes \'Etudes
              Scientifiques},
    VOLUME = {128},
      YEAR = {2018},
     PAGES = {219--397},
      ISSN = {0073-8301,1618-1913},
   MRCLASS = {14F30},
  MRNUMBER = {3905467},
MRREVIEWER = {Daniel\ Robert\ Gulotta},
       DOI = {10.1007/s10240-019-00102-z},
       URL = {https://doi.org/10.1007/s10240-019-00102-z},
}

@article {Madapusi,
    AUTHOR = {Madapusi Pera, Keerthi},
     TITLE = {The {T}ate conjecture for {K}3 surfaces in odd characteristic},
   JOURNAL = {Invent. Math.},
  FJOURNAL = {Inventiones Mathematicae},
    VOLUME = {201},
      YEAR = {2015},
    NUMBER = {2},
     PAGES = {625--668},
      ISSN = {0020-9910,1432-1297},
   MRCLASS = {14J28 (11G10 14G17 14G35 14K15)},
  MRNUMBER = {3370622},
MRREVIEWER = {G.\ K.\ Sankaran},
       DOI = {10.1007/s00222-014-0557-5},
       URL = {https://doi.org/10.1007/s00222-014-0557-5},
}

@article {ChiarellottoLazdaLiedtke,
    AUTHOR = {Chiarellotto, Bruno and Lazda, Christopher and Liedtke,
              Christian},
     TITLE = {A {N}\'eron-{O}gg-{S}hafarevich criterion for {K}3 surfaces},
   JOURNAL = {Proc. Lond. Math. Soc. (3)},
  FJOURNAL = {Proceedings of the London Mathematical Society. Third Series},
    VOLUME = {119},
      YEAR = {2019},
    NUMBER = {2},
     PAGES = {469--514},
      ISSN = {0024-6115,1460-244X},
   MRCLASS = {14J28 (11G25 14F20 14F30 14G20)},
  MRNUMBER = {3959051},
MRREVIEWER = {Sho\ Tanimoto},
       DOI = {10.1112/plms.12237},
       URL = {https://doi.org/10.1112/plms.12237},
}

@article {Taelman,
    AUTHOR = {Taelman, Lenny},
     TITLE = {K3 surfaces over finite fields with given {$L$}-function},
   JOURNAL = {Algebra Number Theory},
  FJOURNAL = {Algebra \& Number Theory},
    VOLUME = {10},
      YEAR = {2016},
    NUMBER = {5},
     PAGES = {1133--1146},
      ISSN = {1937-0652,1944-7833},
   MRCLASS = {14J28 (11G25 14G15)},
  MRNUMBER = {3531364},
MRREVIEWER = {Sho\ Tanimoto},
       DOI = {10.2140/ant.2016.10.1133},
       URL = {https://doi.org/10.2140/ant.2016.10.1133},
}

@article{FU,
    title={CLASSIFICATION OF POLARIZED SYMPLECTIC AUTOMORPHISMS OF FANO VARIETIES OF CUBIC FOURFOLDS},
    volume={58}, 
    DOI={10.1017/S001708951500004X}, 
    number={1}, 
    journal={Glasgow Mathematical Journal}, 
    author={Fu, Lie}, 
    year={2016}, 
    pages={17–37}
}

@article{LazaZheng,
  title={Automorphisms and periods of cubic fourfolds},
  author={Laza, Radu and Zheng, Zhiwei},
  journal={Mathematische Zeitschrift},
  volume={300},
  number={2},
  pages={1455--1507},
  year={2022},
  publisher={Springer}
}

@article{Matsumoto,
  title={Good reduction criterion for {K3} surfaces},
  author={Matsumoto, Yuya},
  journal={Mathematische Zeitschrift},
  volume={279},
  number={1},
  pages={241--266},
  year={2015},
  publisher={Springer}
}

@article{Zarhin,
    url = {https://doi.org/10.1515/crll.1983.341.193},
    title = {Hodge groups of {K3} surfaces.},
    author = {Yu.G. Zarhin},
    pages = {193--220},
    volume = {1983},
    number = {341},
    journal = {Journal für die reine und angewandte Mathematik},
    doi = {doi:10.1515/crll.1983.341.193},
    year = {1983},
    lastchecked = {2026-08-09}
}

@article{Wintenberger,
     author = {Wintenberger, Jean-Pierre},
     title = {Le corps des normes de certaines extensions infinies de corps locaux ; applications},
     journal = {Annales scientifiques de l'\'Ecole Normale Sup\'erieure},
     pages = {59--89},
     year = {1983},
     publisher = {Elsevier},
     volume = {4e s{\'e}rie, 16},
     number = {1},
     doi = {10.24033/asens.1440},
     mrnumber = {85e:11098},
     zbl = {0516.12015},
     language = {fr},
     url = {https://www.numdam.org/articles/10.24033/asens.1440/}
}
\bibliographystyle{plain}

\end{document}